\documentclass[11pt]{amsart}
\usepackage[utf8]{inputenc}
\usepackage[margin=1.1in]{geometry}
\usepackage{latexsym,amssymb,mathrsfs,eucal}
\usepackage{hyperref}
\usepackage{amsmath,amsthm,amsxtra}
\usepackage{enumitem}
\usepackage{amsfonts}
\usepackage{xcolor}

\numberwithin{equation}{section}

\newtheorem{theorem}{Theorem}[section]

\newtheorem{proposition}[theorem]{Proposition}

\newtheorem{lemma}[theorem]{Lemma}

\newtheorem{definition}[theorem]{Definition}

\theoremstyle{definition}
\newtheorem{remark}[theorem]{Remark}

\newcommand{\R}{\mathbb{R}}
\newcommand{\N}{\mathbb{N}}

\newcommand{\calC}{\mathcal{C}}

\newcommand{\calP}{\mathcal{P}}
\newcommand{\scrD}{\mathscr{D}}
\newcommand{\scrE}{\mathscr{E}}
\newcommand{\scrN}{\mathscr{N}}

\newcommand{\sfF}{\mathsf{F}}
\newcommand{\sfT}{\mathsf{T}}
\newcommand{\leb}{\mathscr{L}}

\newcommand{\Ent}{\mathsf{Ent}}

\newcommand{\Lip}{\mathrm{Lip}}

\renewcommand{\d}{\mathsf{d}}

\newcommand{\lt}{\left}
\newcommand{\rt}{\right}
\newcommand{\bq}{\begin{equation}}
\newcommand{\eq}{\end{equation}}
\newcommand{\pa}{\partial}
\newcommand{\intr}{\int_{\R^d}}
\newcommand{\intrr}{\iint_{\R^d \times \R^d}}
\newcommand{\mc}{\mathcal{C}}
\newcommand{\e}{\varepsilon}

\title[Quantified Overdamped Limit for Vlasov--Fokker--Planck equations]{Quantified overdamped limit for kinetic Vlasov--Fokker--Planck equations with singular interaction forces}
\author[Choi]{Young-Pil Choi$^\dag$}
\address[Young-Pil Choi]{\newline Department of Mathematics\newline
Yonsei University, 50 Yonsei-Ro, Seodaemun-Gu, Seoul 03722, Republic of Korea}
\email{ypchoi@yonsei.ac.kr}

\author[Tse]{Oliver Tse$^\ddag$}
\address[Oliver Tse]{\newline Department of Mathematics and Computer Science\newline
Eindhoven University of Technology, 5600 MB Eindhoven, The Netherlands}
\email{o.t.c.tse@tue.nl}
\date{\today}

\begin{document}

\date{}


\begin{abstract} We establish a quantified overdamped limit for kinetic Vlasov--Fokker--Planck equations with nonlocal interaction forces. We provide explicit bounds on the error between solutions of that kinetic equation and the limiting equation, which is known under the names of aggregation-diffusion equation or McKean--Vlasov equation. Introducing an intermediate system via a coarse-graining map, we quantitatively estimate the error between the spatial densities of the Vlasov--Fokker--Planck equation and the intermediate system in the Wasserstein distance of order 2. We then derive an evolution-variational-like inequality for Wasserstein gradient flows which allows us to quantify the error between the intermediate system and the corresponding limiting equation. Our strategy only requires weak integrability of the interaction potentials, thus in particular it includes the quantified overdamped limit of the kinetic Vlasov--Poisson--Fokker--Planck system to the aggregation-diffusion equation with either repulsive electrostatic or attractive gravitational interactions.
\end{abstract}

\maketitle

\tableofcontents

\section{Introduction}
Our starting point is the following kinetic Vlasov--Fokker--Planck equation:
\begin{align}\label{eq:kinetic-rescaled}
	\partial_t\mu_t^\gamma + \gamma\, v\cdot \nabla_x\mu_t^\gamma + \gamma \nabla_v\cdot\bigl(\mu_t^\gamma (\sfF(x,\rho_t^\gamma)-\gamma v)\bigr) = \gamma^2\Delta_v\mu_t^\gamma\,,\qquad (x,v)\in\R^d\times\R^d\,,
\end{align}
where $(\mu^\gamma_t)_{t\ge 0}\subset\calP(\R^d\times\R^d)$ is a family of probability measures on $\R^d \times \R^d$ with $d \geq 1$, and $\gamma > 0$ is the strength of the linear damping in velocity and diffusion. Here, $\rho_t^\gamma = \int_{\R^d}\mu_t^\gamma(\cdot\,, dv)$ denotes the $x$-marginal
of $\mu_t^\gamma$, representing the positional distribution of points, and $\sfF:\R^d \times \calP(\R^d)\to \R^d$ is the driving force of the system, which in our case arises from an external and/or interaction potential. More precisely, we assume $\sfF$ to take the form
\[
\sfF(x,\rho) = -(\nabla\Phi)(x) - (\nabla K\star\rho)(x)\qquad \text{for\, $(x,\rho)\in\R^d\times \calP(\R^d)$}\,,
\]
where $\Phi:\R^d\to\R$ and $K:\R^d\to \R$ are given functions, and
\[
(\nabla K \star \rho)(x) := \intr \nabla K(x-y)\,\rho(dy)\,.
\]

In this paper, we establish quantitative estimates for the Vlasov--Fokker--Planck equation \eqref{eq:kinetic-rescaled}
 in the overdamped regime, i.e.\ in the regime when $\gamma\gg 1$. Note that equation \eqref{eq:kinetic-rescaled} also includes the classical Vlasov--Poisson--Fokker--Planck system when $\nabla K = \zeta x/|x|^d$, $d\ge 3$. Here, the constant $\zeta$ can be chosen $\zeta=\pm 1$ according to applications in either plasma physics or astrophysics.

At the formal level, when $\gamma \to \infty$, it is expected that the kinetic equation \eqref{eq:kinetic-rescaled} will converge to the following well-known {\em aggregation-diffusion} equation:
\begin{align}\label{eq:diffusion}
	\partial_t\rho_t + \nabla_x\cdot \bigl(\rho_t \sfF(\cdot,\rho_t)\bigr) = \Delta_x \rho_t\,,\qquad x\in\R^d\,.
\end{align}
Equations of the type \eqref{eq:diffusion} appear in various contexts, e.g.\ in biological pattern formation \cite{Ma01}, semiconductor equations \cite{MRS90}, models for granular flows and self-gravitating particles \cite{BT1994,CMV03}, hydrodynamic limits of vortex systems \cite{CLMP1992,CLMP1995}, and even in the theory of combustion \cite{BE1989}. The best known classical example is the Keller--Segel model \cite{KS70}, with $K$ satisfying $\Delta K = \delta_0$, which describes the collective motion of chemotaxis. 

More recently, aggregation-diffusion equations have been studied in the framework of gradient flows with respect to the 2-Wasserstein distance defined below \cite{AGS05,CCY19,JKO1998}. Formally, this means that we can recast \eqref{eq:diffusion} into the form
\begin{align*}
	\left\{\quad
	\begin{aligned}
	\partial_t\rho_t + \nabla_x\cdot (\rho_t\mathsf{v}_t) &= 0\,,\\
	\mathsf{v}_t &= -\nabla_x \delta_\rho\scrE(\rho_t)\,,
	\end{aligned}\right.
\end{align*}
where $\delta_\rho\scrE$ is the $L^2$-derivative of the free energy $\scrE$, which in this case takes the form
\begin{align}\label{eq:free-energy-rho}
	\calP(\R^d)\ni \rho\mapsto \scrE(\rho) := \Ent(\rho|\leb^d) + \intr \Phi \,d\rho + \frac{1}{2}\intr K \star \rho\, d\rho\,.
\end{align}
Here, $\Ent(\mu|\nu)$ is the relative entropy of $\mu$ with respect to $\nu$ (cf.\ \eqref{eq:def_entropy} below). We will indeed exploit the gradient flow structure associated to \eqref{eq:diffusion} to prove our results.

\subsubsection*{Formal derivation of the aggregation-diffusion equation} Let us briefly and formally explain why we would expect equation \eqref{eq:diffusion} to arise from \eqref{eq:kinetic-rescaled} as $\gamma \to \infty$. We begin by writing \eqref{eq:kinetic-rescaled} as
\[
\pa_t \mu_t^\gamma + \gamma v \cdot \nabla_x \mu_t^\gamma + \gamma \nabla_v \cdot (\sfF(x,\rho_t^\gamma)\mu_t^\gamma) = \gamma^2 \nabla_v \cdot (\nabla_v \mu_t^\gamma + v \mu_t^\gamma)\,.
\]
Since the right hand side of the above equation can be rewritten as 
\[
\gamma^2 \nabla_v \cdot (\nabla_v \mu_t^\gamma + v \mu_t^\gamma) = \gamma^2 \nabla_v \cdot \lt(\mu_t^\gamma \nabla_v \log\frac{ \mu_t^\gamma}{\scrN^d}\rt)
\]
with the standard $d$-dimensional normal distribution (or Maxwellian)
\[
	\scrN^d(v) = \frac{1}{(2\pi)^{d/2}} e^{-|v|^2/2}\,,
\]
and it has the order $\gamma^2$, we infer that
\bq\label{formal_1}
\mu_t^\gamma(x,v) \simeq \rho_t^\gamma(x)\scrN^d(v)\qquad\text{for $\gamma \gg 1$\,.}
\eq
On the other hand, if we set $m_t^\gamma := \int_{\R^d} v\mu_t^\gamma(\cdot,dv)$, we then find from \eqref{eq:kinetic-rescaled} that
\begin{align}\label{formal_2}
\left\{\quad
\begin{aligned}
\pa_t \rho_t^\gamma + \gamma \nabla_x \cdot m_t^\gamma &= 0\,,\cr
\pa_t m_t^\gamma + \gamma \nabla_x \cdot \lt( \intr v \otimes v\, \mu_t^\gamma(\cdot,dv) \rt) &= \gamma \rho_t^\gamma \sfF(\cdot,\rho_t^\gamma) - \gamma^2 m_t^\gamma\,.
\end{aligned}\right.
\end{align}
We now use the approximation \eqref{formal_1} to obtain
\[
 \gamma \nabla_x \cdot \lt( \intr v \otimes v\, \mu_t^\gamma(\cdot,dv) \rt) \simeq \gamma \nabla_x \rho_t^\gamma  \qquad\text{for $\gamma \gg 1$\,,}
\]
and this shows that the momentum $m_t^\gamma$ roughly satisfies
\[
\pa_t m_t^\gamma + \gamma \nabla_x \rho_t^\gamma \simeq \gamma \rho_t^\gamma \sfF(\cdot,\rho_t^\gamma) - \gamma^2 m_t^\gamma\,.
\]
We use again the fact $\gamma \gg 1$ to neglect the lowest-order term in the above equation, and obtain:
\[
\gamma m_t^\gamma \simeq \rho_t^\gamma \sfF(\cdot,\rho_t^\gamma) -  \nabla_x \rho_t^\gamma\,.
\]
Inserting this into the continuity equation \eqref{formal_2} yields
\[
\pa_t\rho_t^\gamma + \nabla_x\cdot \bigl(\rho_t^\gamma \sfF(\cdot,\rho_t^\gamma)\bigr) \simeq \Delta_x \rho_t^\gamma\,,
\]
which is our limiting equation \eqref{eq:diffusion}.

\subsubsection*{Motivation of \eqref{eq:kinetic-rescaled} from interacting particle systems}
The kinetic Vlasov--Fokker--Planck equation \eqref{eq:kinetic-rescaled} is closely related to classical Newton dynamics for point particles interacting through the force field $\sfF$ in the presence of noise. More precisely, under suitable assumptions on $\Phi$ and $K$, \eqref{eq:kinetic-rescaled} can be derived from the following system of stochastic differential equations \cite{Sznitman1991}:
\begin{align}\label{sde}
\left\{\quad\begin{aligned}
d X^{\gamma,i}_t &= \gamma\,V^{\gamma,i}_t\,dt\,, \cr
\frac{1}{\gamma}\, dV^{\gamma,i}_t &= -\,\nabla \Phi(X^{\gamma,i}_t)\,dt - \frac1N \sum_{j\neq i}  \nabla K (X^{\gamma,i}_t - X^{\gamma,j}_t)\,dt - \gamma V^{\gamma,i}_t\,dt + \sqrt{2}\,dB^i_t\,,
\end{aligned}\right.
\end{align}
where $X^{\gamma,i}_t \in \R^d$ and $V^{\gamma,i}_t \in \R^d$, $i=1,\dots,N$, denote the position and velocity of $i$-th particle at time $t > 0$, respectively, and $B^1_t,\ldots, B^N_t$ are $N$ independent $d$-dimensional Brownian motions. By considering the mean-field limit $N \to \infty$, the so-called nonlinear McKean process
\begin{align}\label{mckean}
	\left\{\quad\begin{aligned}
	dX_t^\gamma &= \gamma\,V_t^\gamma\,dt\,,\\
	\frac{1}{\gamma}\,dV_t^\gamma &= \sfF(X_t^\gamma,\rho_t)\,dt - \gamma V_t^\gamma\,dt + \sqrt{2}\,dB_t\,,
	\end{aligned}\right.
\end{align}
replaces the system \eqref{sde}, where $\rho_t^\gamma = \text{Law}(X_t^\gamma)$ is the law of $X_t^\gamma$, and the time marginal of the process $(X_t^\gamma,V_t^\gamma)_{t\ge 0}$, $\mu_t^\gamma=\text{Law}(X_t^\gamma,V_t^\gamma)$, satisfies \eqref{eq:kinetic-rescaled} in the sense of distributions.

\subsubsection*{Novelty of the current work} 
The study of the overdamped limit has been of interest since the seminal work of Kramers \cite{Kramers1940}. Kramers formally discussed convergence results for the kinetic Fokker--Planck equation ($K=0$), now known as {\em Smoluchowski--Kramers limit}, by introducing a \textit{coarse-graining} map---we consider the same coarse-graining map in this work (see \eqref{eq:intro_cg} below). In \cite{Nelson1967}, Nelson made Kramer's results rigorous by investigating the corresponding stochastic differential equations \eqref{sde}. He showed that a suitable rescaling of the solution to the Langevin equation converges almost surely to the solution of \eqref{eq:diffusion} with $K=0$. Since then, several related results for the case $K=0$ have been proven, by either using stochastic and asymptotic techniques \cite{Freidlin2004,GN2020,HVW2012,Narita1994}, or by means of variational techniques \cite{DLPS18}. The work \cite{DLPS18} is particularly interesting, since it provides both quantitative estimates in the relative entropy and the 2-Wasserstein distance, where the latter is deduced from the former by means of the Talagrand and log-Sobolev inequalities.

In the presence of an interaction potential, a variational technique is proposed in \cite{DLPS17} to study the rigorous limit from \eqref{eq:kinetic-rescaled} to the equation \eqref{eq:diffusion} under suitable regularity and integrability assumptions on the interaction potential $K$, for instance $K \in \mc^2(\R^d) \cap W^{1,1}(\R^d)$ having globally bounded first and second derivatives. This strategy is based on duality methods and compactness arguments via a coarse-graining map, which yields the overdamped limit without convergence rates. In \cite{EM10,Gou05,PoS00}, the qualitative analysis of overdamped limit from the Vlasov--Poisson--Fokker--Planck system towards the drift-diffusion equation in the absence of the potential $\Phi$ is stuided, i.e. from \eqref{eq:kinetic-rescaled} to \eqref{eq:diffusion} with $K$ given as attractive or repulsive Coulomb potential  and $\Phi \equiv 0$. These works are mainly based on the uniform-in-$\gamma$ entropy estimate combined with the compactness arguments, and thus it is not clear how it can be extended or refined to have quantitative error estimates between solutions to \eqref{eq:kinetic-rescaled} and \eqref{eq:diffusion}. 

To the authors' best knowledge, a quantified overdamped limit for the Vlasov--Fokker--Planck equation with smooth nonlocal forces has only been established very recently in \cite{RSX2020}, while the case with singular nonlocal forces remains open. The main purpose of this study is, therefore, to develop a new methodology for the quantified overdamped limit from from \eqref{eq:kinetic-rescaled} to \eqref{eq:diffusion} that works even with irregular or attractive/repulsive singular interaction potentials. We would also like to emphasize that a well-posedness theory for the Vlasov--Fokker--Planck \eqref{eq:kinetic-rescaled} with a required energy dissipation in a weight Sobolev space is established---this is necessary due to the lack of a well-posedness theory in the literature for the case when the external potential $\Phi$ is unbounded.


The aggregation-diffusion equation \eqref{eq:diffusion} can also be rigorously derived from compressible Euler-type equations via large friction limits \cite{CCT19, Cpre, CG07, GLT17}. Despite these fruitful studies on the overdamped limit at the hydrodynamic level,  there are only a few works on the rigorous limit from the kinetic equation to the associated continuity-type equation. In the absence of diffusion, the large friction limit of a Vlasov equation with nonlocal forces is rigorously studied in \cite{Ja00}, and later in \cite{FS15}, where the assumptions on the potentials used in \cite{Ja00} are relaxed. More recently, a quantified large friction limit of a Vlasov-type equation was established in \cite{CC20}, using a pressureless Euler system as an intermediate system.

\subsubsection*{Notation} Throughout the paper, we adopt a slight abuse of notation by identifying the probability measure with its Lebesgue density, i.e.\ $\mu_t(dxdv) = \mu_t(x,v)\,\leb^{2d}(dxdv)$ for simplicity of presentation. Here, $\leb^m$, $m\in\N$ denotes the Lebesgue measure on $\R^m$. By $\calC_c^\infty(\R^m)$, we mean the space of smooth functions with compact support; $\calC_b(\R^m)$ the space of continuous and bounded functions; and $\calC(I;X)$ the space of continuous maps from an interval $I\subset\R$ to a metric space $X$. By $B_r(x)\subset\R^d$ we mean the ball of radius $r>0$ around $x\in\R^d$. For notational simplicity, we set $B_r:=B_r(0)$.

\subsubsection*{Outline of methodology}
In the current work, we provide a rigorous limit from the kinetic equation \eqref{eq:kinetic-rescaled} to the aggregation-diffusion equation \eqref{eq:diffusion} with explicit bounds in the 2-Wasserstein distance defined below. Our methodology hinges upon three main ingredients:

\begin{enumerate}
	\item[(i)] An auxiliary continuity equation induced by a coarse-graining map;
	\item[(ii)] A stability estimate in the 2-Wasserstein distance for solutions to the continuity equation with different velocity fields;
	\item[(iii)] A priori $\gamma$-independent estimates for solutions to \eqref{eq:kinetic-rescaled} in a weighted Sobolev space.
\end{enumerate}

We would like to emphasize that this strategy enables us to take into account the attractive or repulsive Coulomb interaction potential $K$, i.e.\ with $K$ satisfying $\pm\Delta K = \delta_0$. Thus, our result includes the overdamped limit from Vlasov--Poisson--Fokker--Planck system to the classical Keller--Segel equation. In fact, for purely repulsive interactions, we can go beyond the Newtonian singularity, i.e.\ Riesz potentials $K$ can be considered, see Remark \ref{rmk_sing_ex} below for details. 

\medskip

Let us give a brief outline of our strategy. By introducing the map
\begin{align}\label{eq:intro_cg}
	\Gamma^\gamma:\R^d\times\R^d\to\R^d\times \R^d,\quad\Gamma^\gamma(x,v)=(x+v/\gamma,v)\,,\qquad \gamma> 0\,,
\end{align}
we find that the {\em coarse-grained} quantity
 $\bar\rho^\gamma := (\pi^x \circ\Gamma^\gamma)_\# \mu^\gamma$ satisfies an equation of the form
\begin{align}\label{eq:aux-rho0}
\left\{\quad
\begin{aligned}
\partial_t\bar\rho_t^\gamma + \nabla_x\cdot \bar\jmath_t^\gamma &= \Delta_x\bar\rho_t^\gamma\,,\\
\bar\jmath_t^\gamma(dx) &= \intr \sfF(x-v/\gamma,\rho_t^\gamma)\,(\Gamma^\gamma_\#\mu_t^\gamma)(dxdv)\,.
\end{aligned}\right.
\end{align}
Here $\pi^x:\R^d\times\R^d\to\R^d$, $\pi(x,v)=x$ is the projection into the $x$-component, and for a measure $\nu\in\calP(\R^m)$, $f {}_\# \nu\in\calP(\R^k)$ is the push-forward of $\nu$ by a measurable map $f :\R^m \to \R^k$, $m,k\in\N$.

Notice that the coarse-grained system 
resembles the aggregation-diffusion equation \eqref{eq:diffusion}---the essential difference being the drift term---and plays the role of an intermediate system between the kinetic equation \eqref{eq:kinetic-rescaled} and the limiting equation \eqref{eq:diffusion}.  Therefore, the convergence of $\bar\rho^\gamma$ towards $\rho$ is related to the stability of solutions to Fokker--Planck equations with different drifts.

For the quantitative error estimates of solutions, we use the $2$-Wasserstein distance, where for any $p\in[1,\infty)$, the $p$-Wasserstein distance is defined by
\[
	\d_p(\mu, \nu) := \inf_{\pi \in \Pi(\mu, \nu)} \lt(\iint_{\R^m\times\R^m} |x- y|^p \,\pi(dxdy)  \rt)^{1/p}
\]
for any Borel probability measures $\mu$ and $\nu$ on $\R^m$, $m\in\N$, where $\Pi(\mu,\nu)$ is the set of all probability measures on $\R^m \times \R^m$ with first and second marginals $\mu$ and $\nu$, respectively, i.e.\ for all $\varphi\,, \psi \in \mc_b(\R^m)$:
\[
	\iint_{\R^m\times\R^m} (\varphi(x) + \psi(y))\, \pi(dx dy) = \int_{\R^m} \varphi(x)\,\mu(dx) + \int_{\R^m} \psi(y)\,\nu(dy)\,.
\]
We denote by $\mathcal{P}_2(\R^m)$ the set of probability measures in $\R^m$ with bounded second moment. Then $\mathcal{P}_2(\R^m)$ is a complete metric space endowed with the $2$-Wassertein distance $\d_2$ (see \cite{AGS05} or \cite{Vil03} for a more detailed discussion).

By employing the $2$-Wasserstein distance, we first estimate the $\d_2$-distance between $\rho^\gamma$ and $\bar\rho^\gamma$. By our definition of $\bar\rho^\gamma$, we easily obtain the quantitative bound of terms of $\gamma$ and the kinetic energy of $\mu^\gamma$. Indeed, we show in Lemma~\ref{lem_rho2} below that
\[
	\d_2^2(\bar\rho_t^\gamma,\rho_t^\gamma) \le \frac{1}{\gamma^2}\intrr |v|^2\,\mu_t^\gamma(dxdv)\,.
\]
The main difficulty arises from the error estimate between $\bar\rho^\gamma$ and $\rho$, the solution to the limiting equation \eqref{eq:diffusion}. For this, we introduce in Section~\ref{sec:modulated} a {\em modulated interaction energy} 
\[
\scrD_K(\mu,\nu):=\intrr K(x-y)[\mu-\nu](dy)[\mu-\nu](dx)\qquad \text{for $\mu,\nu \in \mathcal{P}(\R^d)$\,,}
\]
and derive in Section~\ref{sec:EVI} a stability estimate of the form 
\bq\label{EVI0}
\frac{1}{2}\frac{d}{dt}\d_2^2(\bar\rho_t^\gamma,\rho_t) \le  \lambda \d_2^2(\bar\rho_t^\gamma,\rho_t) - 2\scrD_K(\bar\rho_t^\gamma,\rho_t) + \frac{1}{2}\|\mathsf{e}_t^\gamma\|_{L^2(\bar\rho_t^\gamma)}^2,
\eq
for some $\lambda>0$, where $\mathsf{e}_t^\gamma$, explicitly given in Lemma~\ref{lem_ee}, is related to the error between the drift fields in  \eqref{eq:diffusion} and \eqref{eq:aux-rho0}. 

We notice that the modulated interaction energy $\scrD_K$ has a positive sign when the interaction potential $K$ gives rise to positive definite kernels---purely repulsive interactions fall within this class of potentials. Thus for the repulsive interaction potentials, we only need to estimate the error term $\|\mathsf{e}_t^\gamma\|_{L^2(\bar\rho_t^\gamma)}$. On the other hand, in the presence of an attractive interaction potential, $\scrD_K$ may take negative values, and hence we require $|\scrD_K(\mu,\nu)|$ to be controlled by $\d_2^2(\mu,\nu)$. This can be achieved for appropriate interaction potentials $K$ as discussed in Section~\ref{sec:modulated} (see Theorem~\ref{thm:DK-interaction}).

%
%
%

In the regular interaction potential case, we do not require higher-order regularity of solutions $\mu^\gamma$ to \eqref{eq:kinetic-rescaled}. However, for taking into account singular interaction potentials, we will need additional regularity of the force field $\sfF$. This requires some uniform-in-$\gamma$ estimates of solutions $\mu^\gamma$. For this, we introduce a Sobolev space $W_H^{k,p}$ weighted by the exponential of the Hamiltonian $H(x,v) := \Phi(x) + |v|^2/2$ (cf.\ Section \ref{sec:assumption-results}). Our careful analysis of solutions $\mu^\gamma$ in this space makes the interaction term $\nabla K \star \rho^\gamma$ Lipschitz continuous and bounded, uniformly in $\gamma$, even if the potential $K$ has a strong singularity at the origin. For details, see Section \ref{sec:uniform-gamma}.

\medskip

Unfortunately, we were unable to apply the above strategy to deduce stability estimates in the 2-Wasserstein distance when $\nabla K$ is only assumed to be bounded. The reason lies in our inability to relate the modulated interaction energy $\scrD_K$ to $\d_2^2$. Nevertheless, we are able to obtain stability estimates with respect to the bounded Lipschitz distance $\d_{\text{BL}}$ by other means. This is addressed in Section~\ref{sec:bdd}.

\subsection{Main assumption and results}\label{sec:assumption-results} In this part, we list our main results depending on the regularity of the interaction potential $K$ in the force field $\sfF$. 

Concerning the potential function $\Phi$, we assume throughout this paper that $0 \leq \Phi \in \Lip_{loc}(\R^d)$, and that
\begin{enumerate}
\item[($\textbf{A}_\Phi^1$)] there exists $c_\Phi>0$ such that
\[
	|\nabla \Phi(x)|\le c_\Phi(1 + |x|)\quad\text{for all $x\in\R^d$}\quad\text{and}\quad \|\nabla\Phi\|_{\Lip}\le c_\Phi\,;
\]
\item[($\textbf{A}_\Phi^2$)] for any $r\in[1,\infty)$:
\[
	c_{\Phi,r}:= \sup_{x\in\R^d}|\nabla\Phi(x)|^r e^{-\Phi(x)}<\infty\,.
\]
\end{enumerate}
Here, we denoted by $\Lip(\R^d)$ the set of Lipschitz functions on $\R^d$ and
\[
\|\phi\|_{\Lip} := \sup_{x \neq y} \frac{|\phi(x) - \phi(y)|}{|x-y|}\qquad\text{for $\phi \in \Lip(\R^d)$\,.}
\]

Note that the assumptions ($\textbf{A}_\Phi$) above allow us to consider both bounded and unbounded external potentials with at most quadratic growth at infinity, i.e.\ we can consider the confinement potential $\Phi(x) = |x|^2/2$. 

\begin{definition}\label{def:weak-solution}
	A continuous curve $\mu\in\calC([0,T];\calP_2(\R^d\times\R^d))$ is called a weak solution to the Cauchy problem for \eqref{eq:kinetic-rescaled}, with initial datum $\mu_0\in \calP_2(\R^d\times\R^d)$, if $\mu|_{t=0} = \mu_0$, and \eqref{eq:kinetic-rescaled} holds in the following sense: for every $s,t\in(0,T)$, and every $\varphi\in\calC_c^\infty(\R^d\times\R^d)$,
	\begin{align*}
		\langle \varphi,\mu_t\rangle - \langle \varphi,\mu_s\rangle = \gamma\int_s^t \intrr \Bigl(v\cdot \nabla_x\varphi + \bigl(\sfF(x,\rho_r^\gamma)-\gamma v\bigr)\cdot \nabla_v\varphi  + \gamma \Delta_v \varphi\Bigr)\,\mu_r(dxdv)\,dr\,.
	\end{align*}

Similarly, we call a continuous curve $\rho\in \calC([0,T];\calP_2(\R^d))$ a weak solution to the Cauchy problem \eqref{eq:diffusion}, with initial datum $\rho_0\in\calP_2(\R^d)$, if $\rho|_{t=0}=\rho_0$, and \eqref{eq:diffusion} holds in the following sense: for every $s,t\in(0,T)$, and every $\varphi\in\calC_c^\infty(\R^d)$,
	\begin{align*}
		\langle \varphi,\rho_t\rangle - \langle \varphi,\rho_s\rangle = \int_s^t \intr \Bigl(\sfF(x,\rho_r)\cdot \nabla_x\varphi  + \Delta_x \varphi\Bigr)\,\rho_r(dx)\,dr\,.
	\end{align*}
\end{definition}

As for the limiting equation, we consider an additional notion of solution, which allows us to make use of the Wasserstein (or Otto) gradient flow structure of \eqref{eq:diffusion}.

\begin{definition}\label{def:regular-solution}
	Let $\scrE$ be the free energy given in \eqref{eq:free-energy-rho}. We call a continuous curve $\rho\in \calC([0,T];\calP_2(\R^d))$ an $\scrE$-regular solution to the Cauchy problem \eqref{eq:diffusion}, with initial datum $\rho_0\in\calP_2(\R^d)\cap\text{dom}(\scrE)$, if it is a weak solution in the sense of Definition~\ref{def:weak-solution}, satisfying additionally $\rho\in L^1((0,T); W^{1,1}(\R^d))$, and
	\[
		\sup_{t\in[0,T]} \scrE(\rho_t) + \sup_{t\in [0,T]} \|\sfF(\cdot,\rho_t)\|_{L^\infty} + \int_0^T \intr \frac{|\nabla\rho_t|^2}{\rho_t}\,d\leb^d\,dt <\infty\,.
	\]
\end{definition}

For any nonnegative Radon measure $\nu$ and probability measure $\mu$ in $\R^m$, $m\in\N$, we set
\begin{align}\label{eq:def_entropy}
\Ent(\mu|\nu) := \begin{cases}
 	\displaystyle\int_{\R^m} h(d\mu/d\nu)\,d\nu &\text{if $\mu\ll \nu$\,,} \\[2mm]
 	+\infty &\text{otherwise,}
 \end{cases}
\end{align}
where $h(s) := s\log s$ for any $s>0$, extended to $h(0):=0$.

\medskip

Our first main result concerns the case $\nabla K \in L^\infty(\R^d) \cap \Lip(\R^d)$. Under this regularity assumption on $K$, the global-in-time well-posedness of weak solutions $\mu^\gamma$ to \eqref{eq:kinetic-rescaled}, for each $\gamma>0$, follows immediately from the standard existence theory for the corresponding nonlinear McKean process \eqref{mckean} \cite{BCC11,Sznitman1991}, and similarly for the weak solution $\rho$ to \eqref{eq:diffusion} \cite{CCY19}.

\begin{theorem}\label{thm_reg} Let $T>0$ and assume that the interaction potential K satisfies $\nabla K \in L^\infty(\R^d) \cap \Lip(\R^d)$. 
Moreover, let the family of initial data $(\mu_0^\gamma)_{\gamma\ge 1}\subset \calP_2(\R^d\times\R^d)$ satisfy
\[
	\sup_{\gamma\ge 1}\,\Ent(\mu_0^\gamma|\scrN^{2d}) <\infty\,.
\]
Then the associated family of unique weak solutions $(\mu^\gamma)_{\gamma\ge 1}\in \calC([0,T];\calP_2(\R^d\times\R^d))$ of \eqref{eq:kinetic-rescaled} satisfies
\[
	\sup_{\gamma\ge 1}\, \sup_{t\in[0,T]}\Ent(\mu_t^\gamma|\scrN^{2d})  <\infty\,.
\]

Moreover, if $\rho \in \mc([0,T];\calP_2(\R^d))$ is the unique $\scrE$-regular solution of \eqref{eq:diffusion} in the sense of Definition~\ref{def:regular-solution}, with initial datum $\rho_0\in\calP_2(\R^d)\cap\text{dom}(\scrE)$, and $\rho_t^\gamma = \pi^x_\#\mu_t^\gamma$ is the first marginal of $\mu_t$, then
	\[
		\sup_{t\in[0, T]}\d_2^2(\rho^\gamma_t, \rho_t) \leq C\lt(\d_2^2(\rho_0^\gamma, \rho_0) + \frac{1}{\gamma^2}\rt),
	\]
	for some constant $C>0$ independent of $\gamma\ge 1$.
%
\end{theorem}

 The first part of Theorem~\ref{thm_reg} provides uniform-in-time $\gamma$-independent moment estimates for the family of solutions $(\mu_t^\gamma)_{\gamma\ge 1}$ obtained from standard existence theory mentioned above, which is essential for establishing the second part of the statement---the constant $C>0$ in the theorem depends explicitly on the moment bounds (see Section~\ref{sec:bL}).



\medskip

Next, we consider singular interaction potentials, e.g.\ $\nabla K \in L^q_{loc}(\R^d)$. As mentioned above, we require more regular solutions $\mu^\gamma$ of the kinetic equation \eqref{eq:kinetic-rescaled}. For $p \in [1,\infty)$, we denote by $L_H^p:=L^p_H(\R^d \times \R^d)$ the space of measurable functions whose $p$-th powers are weighted by the exponential of the Hamiltonian $H(x,v) = \Phi(x) + |v|^2/2$, and equipped with the norm
\[
\|\mu\|_{L^p_H} := \lt(\intrr \mu^p e^{(p-1)H}\,d\leb^{2d} \rt)^{1/p}.
\]
For any $k\in\N$, $W^{k,p}_{x,H}:=W^{k,p}_{x,H}(\R^d \times \R^d)$ represents $L^p_H$ Sobolev space of $k$-th order in $x$, i.e.
\[
\|\mu\|_{W^{k,p}_{x,H}}:= \lt(\sum_{|\alpha| \leq k} \intrr |\nabla_x^\alpha \mu|^p e^{(p-1)H}\,d\leb^{2d}   \rt)^{1/p}.
\]
We also denote by $W^{k,p}_x := W^{k,p}_{x,0}$, i.e.
\[
\|\mu\|_{W^{k,p}_{x}}:= \lt(\sum_{|\alpha| \leq k} \intrr |\nabla_x^\alpha \mu|^p \,d\leb^{2d}   \rt)^{1/p}.
\]

\medskip

We now state theorems that hold for singular interaction forces.

\begin{theorem}\label{thm:well-posedness}
	Let $T>0$ and suppose that the interaction potential $K$ satisfies
	\bq\label{K_condi}
\|\nabla K\|_{L^q(B_{2R})} + \|\nabla K\|_{W^{1,\infty}(\R^d\setminus B_R)} <\infty \quad\text{for some $R>0$ and $q\in(1,\infty]$.}
		\eq
	Let the family of initial data $(\mu_0^\gamma)_{\gamma\ge 1}\subset (\calP_2\cap W_{x,H}^{1,p})(\R^d\times\R^d)$ with $p\ge \max\{2,q/(q-1)\}$ satisfy
\[
	\sup_{\gamma\ge 1}\, \Bigl\{\Ent(\mu_0^\gamma|\scrN^{2d}) + \|\mu_0^\gamma\|_{W_{x,H}^{1,p}}\Bigr\} <\infty\,.
\] 
	There exists a positive time $T_p\in(0,T]$, and a family of unique weak solutions
\[
	(\mu^\gamma)_{\gamma\ge 1}\in \calC([0,T_p];\calP_2(\R^d\times\R^d))\cap L^\infty([0,T_p];W^{1,p}_{x,H}(\R^d \times \R^d))\quad\text{to\; \eqref{eq:kinetic-rescaled}}
\]
satisfying
\[
	\sup_{\gamma\ge 1}\, \sup_{t\in[0,T_p]}\Bigl\{\Ent(\mu_t^\gamma|\scrN^{2d}) + \|\mu_t^\gamma\|_{W_{x,H}^{1,p}}\Bigr\} <\infty\,.
\]
\end{theorem}

We note that Theorem \ref{thm:well-posedness} includes the Vlasov--Poisson--Fokker--Planck system (i.e.\ \eqref{eq:kinetic-rescaled} with $\Phi\equiv 0$) with either repulsive electrostatic or attractive gravitational interactions:
\[
	\nabla K(x) = \pm \frac{x}{|x|^d},\qquad d\ge 1\,.
\]
We refer to \cite{CS95, PS00, Vic91} for the global-in-time existence of weak solutions for the Vlasov--Poisson--Fokker--Planck system. For interaction potentials that are less singular than Coulomb (see \eqref{less_sing} below), we refer to \cite{CCS19} for the existence of weak solutions.

On the other hand, our result applies to the case with possibly unbounded external potentials $\Phi$, and provides bounds in the weighted Sobolev space $W_{x,H}^{1,p}$. More importantly, these bounds are uniform in $\gamma\ge1$, which plays an essential role in our quantified error estimate. In a nutshell, the existence and uniqueness result follows from a truncation procedure, for which the truncated equation admits unique global-in-time smooth solutions (cf.\ \cite{Bou93,Degond86,VO90}), and by a careful analysis when removing the truncation, while maintaining the $\gamma$-independent bounds.

\medskip

The following result summarizes the quantified error estimate with respect to the 2-Wasserstein distance for various types of singular interaction potentials $K$.

\begin{theorem}\label{thm:stability-limit} Let the assumptions of Theorem~\ref{thm:well-posedness} be satisfied with $p>\max\{d,q/(q-1)\}$ and $T_p>0$ be as given in Theorem~\ref{thm:well-posedness}, and $(\mu^\gamma)_{\gamma\ge 1}$ be the family of weak solutions to \eqref{eq:kinetic-rescaled} given by that theorem.

Suppose that the interaction potential $K$ additionally satisfies one of the following conditions:
\begin{enumerate}[label=(\roman*)]
	\item (Weakly singular) $\nabla K|_{B_{2R}}\in W^{1,1}(B_{2R})$ for some $R<\infty$,
	\item (Purely repulsive) $K$ is positive definite, i.e.\ $\scrD_K(\mu,\nu)\ge 0$ for any $\mu,\nu\in\calP(\R^d)$, or
	\item (Attractive Newtonian) $K$ is given by the Newtonian potential, i.e.\ $\Delta K = \delta_0$, where $\delta_0$ denotes the Dirac measure on $\R^d$ giving unit mass to the origin.
\end{enumerate}

If $\rho \in \calC([0,T_p];\calP_2(\R^d))$ is the unique $\scrE$-regular solution of \eqref{eq:diffusion} in the sense of Definition~\ref{def:regular-solution}, with initial datum $\rho_0\in \calP_2(\R^d)\cap\text{dom}(\scrE)$, then 
 \[
		\sup_{0 \leq t \leq T_p}\d_2^2(\rho^\gamma_t, \rho_t) \leq C\lt(\d_2^2(\rho^\gamma_0, \rho_0) + \frac{1}{\gamma^2}\rt),
	\]
	for some constant $C>0$ independent of $\gamma\ge 1$.
\end{theorem}

Several remarks regarding Theorem~\ref{thm:stability-limit} are in order:

\begin{remark}
	The local-in-time existence of $\scrE$-regular solutions $\rho$ to \eqref{eq:diffusion} can be deduced from the uniform-in-$\gamma$ estimates obtained in Theorem~\ref{thm:well-posedness} via compactness arguments. The uniqueness of such solutions for singular interactions $K$ is, however, a subtle business (see e.g.\ \cite{CLM2014} for the Newtonian potential). Since the well-posedness of the limiting equation is not the main focus of our work, we refer the reader to \cite{CCY19}. That being said, in the absence of uniqueness, the estimate in Theorem~\ref{thm:stability-limit} still holds true, though only for a subsequence of $\gamma\ge1$.
\end{remark}

\begin{remark}
	The assumption $p>d$ in Theorem~\ref{thm:stability-limit} is only necessary to obtain uniform $L^\infty$-bounds via Sobolev embeddings required for relating the modulated interaction energy $\scrD_K$ with the 2-Wasserstein distance $\d_2^2$, and may be removed if one knows a priori that $\gamma$-independent $L^\infty$-bounds hold true.
\end{remark}

\begin{remark}\label{rmk_sing_ex} 
The assumption on $K$ in Theorem \ref{thm:stability-limit} allows us to consider the Riesz potentials, 
\[
	K(x) = \frac{1}{|x|^\alpha} \qquad \mbox{with } -1 \leq \alpha < d-1
\]
for the purely repulsive cases, and
\[
	K(x) = -\frac{1}{|x|^\alpha} \qquad \mbox{with } -1 \leq \alpha \leq d-2
\]
for the purely attractive cases ($\alpha=d-2$ being the Coulomb potential). Indeed, for the repulsive case, we only need to check the condition \eqref{K_condi}. Since we can always find $q \in (1,\infty]$ such that $(\alpha+1)q< d$ when $\alpha <  d-1$, we have
\[
|\nabla K (x)| =  \frac{|\alpha|}{|x|^{\alpha+1}} \in L^q(B_{2R})\qquad\text{for some $R>0$.}
\]
It is readily seen that $\nabla K \in W^{1,\infty}(\R^d \setminus B_R)$. For the attractive case, we need to verify the conditions \eqref{K_condi} and Theorem \ref{thm:stability-limit} (i) or (iii). Similarly as before, condition \eqref{K_condi} imposes $\alpha <  d-1$. If $\alpha = d-2$, then $K$ satisfies the attractive Poisson equation (up to a constant), and we are done. On the other hand, if $- 1 \leq \alpha < d-2$, then 
\[
|\nabla^2 K (x)| \leq \frac{C}{|x|^{\alpha+2}} \in L^1(B_{2R})
\]
for some constant $C>0$, where the integrability holds since $\alpha+2 < d$, thus validating condition Theorem \ref{thm:stability-limit} (i). In particular, we can consider interaction potentials $K$ that are made up of combinations of repulsive and attractive potentials, satisfying
\bq\label{less_sing}
	|\nabla K(x)| \leq \frac{C}{|x|^\alpha} \quad \mbox{and} \quad |\nabla^2 K(x)| \leq \frac{C}{|x|^{\alpha+1}}  \qquad \mbox{with } 0 \leq \alpha < d-1\,.
\eq
\end{remark}



\medskip

The final result in the present work is a result concerning the case when $\nabla K$ is bounded and does not possess any additional regularity. The stability estimate holds in the bounded Lipschitz norm defined, for any finite measure $\mu$ and $\nu$ on $\R^d$, by
\[
	\d_{\text{BL}}(\mu,\nu):= \sup\Bigl\{ |\langle f,\mu-\nu\rangle|\,:\, \|f\|_{\Lip}\le 1,\; \|f\|_{L^\infty}\le 1\Bigr\}.
\]

\begin{theorem}\label{thm_wreg} Let $\nabla K\in L^\infty(\R^d)$ and $\Phi$ satisfy additionally $e^{-\Phi}\in L^1(\R^d)$. Further, let $(\mu_t^\gamma)_{t\in[0,T]}$ be a family of weak solutions to \eqref{eq:kinetic-rescaled} satisfying
	\[
		\sup_{\gamma \ge 1}\sup_{t\in[0,T]}\lt\{\|\nabla_x\mu_t^\gamma\|_{L_H^2} + \intrr |v|^2\, \mu_t^\gamma(dxdv)\rt\} <\infty\,.
	\]
If $(\rho_t)_{t\in[0,T]}$ is the unique $\scrE$-regular solution of \eqref{eq:diffusion}, then
\[
\sup_{t\in[0,T]} \d_{\text{BL}}^2(\rho^\gamma_t, \rho_t) \le C\lt(\Ent\bigl((\pi^x\circ \Gamma^\gamma)_\#\mu_0^\gamma\big|\rho_0\bigr) + \frac{1}{\gamma^2}\rt),
\]
for some constant $C>0$ independent of $\gamma\ge 1$.
%
%
%
\end{theorem}

\subsubsection*{Outline of the paper}The rest of this paper is organized as follows. In Section \ref{sec:IS}, we introduce the Fokker--Planck-type intermediate system via the coarse-graining map, and provide some quantitative bounds and regularity estimates of solutions to that system related to the kinetic equation \eqref{eq:kinetic-rescaled}. In particular, we present the uniform-in-$\gamma$ estimate of second moment of $\mu^\gamma_t$ in velocity. Section \ref{sec:2WEV} is devoted to the derivation of the stability estimate \eqref{EVI0}. We then consider the regular case, i.e.\ when $\nabla K$ is bounded and Lipschitz, in Section \ref{sec:bL}, and show the quantitative error estimate between $\rho_t^\gamma$ and $\rho_t$ in the $2$-Wasserstein distance, resulting in Theorem~\ref{thm_reg}. For the singular interaction potentials, we will need a higher-order regularity of solutions. Section \ref{sec:uniform-gamma} provides the well-posedness result for the kinetic equation \eqref{eq:kinetic-rescaled} and uniform-in-$\gamma$ estimates of solutions in the weighted Sobolev space $W_{x,H}^{1,p}$ (cf.\ Theorem~\ref{thm:well-posedness}). In Section \ref{sec:SIP}, we combine the stability estimate \eqref{EVI0} and the uniform-in-$\gamma$ estimates of solutions obtained in Section~\ref{sec:uniform-gamma} to quantify the overdamped limit for the kinetic equation \eqref{eq:kinetic-rescaled} with singular interaction potentials (cf.\ Theorem~\ref{thm:stability-limit}). In Section \ref{sec:bdd}, we provide a different approach to deal with the case when $\nabla K$ is only assumed to be bounded, resulting in the proof of Theorem~\ref{thm_wreg}. 
Finally, we discuss the solvability for our kinetic equation \eqref{eq:kinetic-rescaled}, which makes all the {\it a priori} estimates for the overdamped limit presented in Section~\ref{sec:uniform-gamma} completely rigorous, in Appendix \ref{app:WP}.

%
%
%
%
\subsection*{Acknowledgments}
The work of Y.-P.C.\ is supported by NRF grant (No.\ 2017R1C1B2012918) and Yonsei University Research Fund of 2019-22-0212 and 2020-22-0505. O.T.\ acknowledges support from NWO Vidi grant 016.Vidi.189.102, ``Dynamical-Variational Transport Costs and Application to Variational Evolutions" and thanks Mitia Duerinckx, Mark Peletier and Upanshu Sharma for insightful discussions in the early stages of this manuscript.

\section{An intermediate system}\label{sec:IS}

\subsection{The coarse-graining map}\label{sec:CG-map}

We introduce a transformation of the system via the following coarse-graining map. For each $\gamma>0$, we set
\[
	\Gamma^\gamma:\R^d\times\R^d\to\R^d\times \R^d,\qquad \Gamma^\gamma(x,v)=(x+v/\gamma,v)\,,
\]
which is a diffeomorphism with $\det(D\Gamma^\gamma(x,v))=1$ for all $(x,v)\in\R^d\times\R^d$, i.e.\ $\Gamma^\gamma$ is volume preserving, with its inverse given by $\Gamma^{-\gamma}(\zeta,v):=(\zeta-v/\gamma,v).$

\begin{remark}
    If $\mu^\gamma\ll\leb$, then the change of variable formula gives
    \[
		\frac{d(\Gamma^\gamma_\#\mu^\gamma)}{d\leb}(x,v) = \mu^\gamma\circ \Gamma^{-\gamma}(x,v) = \mu^\gamma(x-v/\gamma,v)\,.
    \]
\end{remark}

In the following lemma, we show the relative entropy estimate of $\mu^\gamma$ with respect to $\Gamma^{-\gamma}_\#\scrN^{2d}$, whose proof is postponed to Appendix~\ref{app:uniform-gamma} (cf.\ Theorem~\ref{thm:entropy-linear}). The lemma provides, in particular, a uniform in $\gamma$ bound on the second moment of $\mu^\gamma$ in velocity, see Remark \ref{rmk:2m} below.

\begin{lemma}\label{lem:g-estimate}
Let $\mu^\gamma$ be a solution of \eqref{eq:kinetic-rescaled} with sufficient regularity and integrability. Then for $\gamma\ge 1$, there exist constants $\alpha,\beta,\lambda>0$, independent of $\gamma\ge 1$, such that
    \[
        \Ent(\mu_t^\gamma|\Gamma^{-\gamma}_\#\scrN^{2d}) + \alpha\gamma^2\int_0^t\intrr |\nabla_v \log \mu_r^\gamma + v|^2\,d\mu_r^\gamma\,dr \le M(\mu_0^\gamma,t)\,,
    \]
    with $M(\mu_0^\gamma,t) = \Ent(\mu_0^\gamma|\Gamma^{-\gamma}_\#\scrN^{2d})e^{\lambda t} + (\beta/\lambda)(e^{\lambda t}-1)$.
\end{lemma}

%

\begin{remark}\label{rmk:2m}
	Note that if $\sup_{\gamma\ge 1}\Ent(\mu_0^\gamma|\scrN^{2d}) <\infty$, then the estimate provided in Lemma~\ref{lem:g-estimate} yields the following uniform-in-$\gamma$ bound:
	\[
    	\sup_{\gamma\ge 1}\sup_{t\in[0,T]}\Ent(\mu_t^\gamma|\scrN^{2d}) <\infty\,.
    \]
    
    To justify this claim, we make use of the Donsker--Varadhan dual variational characterization for the relative entropy \cite[Proposition~1.4.2]{DE97}, which states that for any bounded measurable function $\varphi:\R^{2d}\to\R$, and any $\vartheta\in\calP(\R^{2d})$,
	\[
		\log\lt( \intrr e^{\varphi} d\vartheta\rt) = \sup \lt\{ \intrr \varphi\,d\nu - \Ent(\nu|\vartheta)\,:\, \nu\in \calP(\R^{2d}),\; \Ent(\nu|\vartheta)<\infty\rt\}.
	\]
	This characterization can be generalized to also hold for any nonnegative measurable functions $\varphi:\R^{2d}\to[0,\infty]$ (see \cite[Lemma B.1]{JHMTpre}). Applying the above characterization to $\vartheta = \Gamma^{-\gamma}_\#\scrN^{2d}$ and $\varphi(x,v) = (|x+v/\gamma|^2 + |v|^2)/4$, we deduce a moment estimate for any $\nu\in \calP(\R^{2d})$ with $\Ent(\nu|\Gamma^{-\gamma}_\#\scrN^{2d})<\infty$:
	\begin{align}\label{eq:entropy-moment}
    	d\log 2 \ge \frac14 \intrr |x+v/\gamma|^2 + |v|^2 \,d \nu - \Ent(\nu|\Gamma^{-\gamma}_\#\scrN^{2d})\,,
    \end{align}
	where we use the fact that
    \[
        \log\lt( \intrr e^{\frac{|x+v/\gamma|^2+|v|^2}{4}} d(\Gamma^{-\gamma}_\#\scrN^{2d})\rt) = \log\left(\intrr e^{\frac{|x|^2+|v|^2}{4}} d\scrN^{2d}\right) = d\log 2\,.
    \]

	Now, since
	\begin{align*}
		\Ent(\mu_t^\gamma|\scrN^{2d}) &= \Ent(\mu_t^\gamma|\leb^{2d}) + \frac{1}{2}\intrr |x|^2 + |v|^2\,d\mu_t^\gamma - d\log(2\pi) \\
		&\le \Ent(\mu_t^\gamma|\Gamma^{-\gamma}_\#\scrN^{2d}) + c_1\intrr |x+v/\gamma|^2 + |v|^2\,d\mu_t^\gamma
	\end{align*}
	for some constant $c_1>0$ independent of $\gamma\ge 1$, we obtain from \eqref{eq:entropy-moment} the estimate
	\[
		\Ent(\mu_t^\gamma|\scrN^{2d}) \le c_2\,\Ent(\mu_t^\gamma|\Gamma^{-\gamma}_\#\scrN^{2d}) + c_3
	\]
	for some constants $c_2,c_3>0$, independent of $\gamma\ge 1$. In a similar fashion, one deduces that
	\[
		\,\Ent(\mu_t^\gamma|\Gamma^{-\gamma}_\#\scrN^{2d}) \le c_4\,\Ent(\mu_t^\gamma|\scrN^{2d}) + c_5
	\]
	for appropriate constants $c_4,c_5>0$ that are independent of $\gamma\ge 1$.	
\end{remark}

\subsection{The intermediate system}

Since $\Gamma^\gamma$ is a smooth map for each $\gamma>0$, it induces a map from $\calC_c^\infty(\R^d)\to\calC_c^\infty(\R^d\times\R^d)$. Indeed, for any $\varphi\in\calC_c^\infty(\R^d)$, we set $\varphi^\gamma(x,v):=(\varphi\circ \Gamma_1^\gamma)(x,v) = \varphi(x+v/\gamma)$, where $\Gamma_1^\gamma:=\pi^x\circ\Gamma^{\gamma}$. It is not difficult to see that $\varphi^\gamma\in\calC_c^\infty(\R^d\times\R^d)$. Using the identities
\[
	D_x\varphi^\gamma = (D_x\varphi)\circ\Gamma_1^\gamma\,,\qquad D_v\varphi^\gamma = (1/\gamma)(D_x\varphi)\circ\Gamma_1^\gamma\,,\qquad \Delta_v\varphi^\gamma = (1/\gamma^2)(\Delta_x\varphi)\circ\Gamma_1^\gamma\,,
\]
%
we deduce the following equalities:
\begin{align*}
	&\intrr\Bigl(\partial_t\varphi^\gamma + \gamma\, \langle v,\nabla_{x}\varphi^\gamma\rangle + \gamma\,\langle \sfF(x,\rho_t^\gamma)-\gamma v,\nabla_{v}\varphi^\gamma\rangle + \gamma^2\Delta_v\varphi^\gamma\Bigr)\,\mu_t^\gamma(dxdv) \\
	&\hspace{2em}= \intrr\Bigl((\partial_t\varphi)\circ\Gamma_1^\gamma + \langle \sfF(x,\rho_t^\gamma),(\nabla_x\varphi)\circ\Gamma_1^\gamma\rangle + (\Delta_x\varphi)\circ\Gamma_1^\gamma\Bigr)\,\mu_t^\gamma(dxdv) \\
	&\hspace{2em}= \intr\bigl(\partial_t\varphi + \Delta_x\varphi\bigr)\,((\pi^x\circ\Gamma^\gamma)_\#\mu_t^\gamma)(dx) + \intrr \langle\sfF(x - v/\gamma,\rho_t^\gamma), \nabla_x\varphi\rangle\,(\Gamma^\gamma_\#\mu_t^\gamma)(dxdv) \,.
\end{align*}
Thus, denoting $\bar\rho^\gamma := \Gamma^{\gamma}_1 {}_\# \mu^\gamma$ we find from \eqref{eq:kinetic-rescaled} that $\bar\rho^\gamma$ satisfies
\begin{align*}
	\int_0^t\intr\Bigl(\partial_s\varphi + \Delta_x\varphi\Bigr)\,d\bar\rho_s^\gamma\,ds + \int_0^t\intrr \langle\sfF(x - v/\gamma,\rho_s^\gamma), \nabla_x\varphi\rangle\,(\Gamma^\gamma_\#\mu_s^\gamma)(dxdv)\,ds = \intr \varphi(0,\cdot)\,d\bar\rho_0^\gamma
\end{align*}
for any $\varphi\in\calC_c^\infty(\R^d)$. In particular, $(\bar\rho_t^\gamma)$ satisfies the equation
\begin{align}\label{eq:aux-rho}
	\left\{\quad
	\begin{aligned}
		\partial_t\bar\rho_t^\gamma + \nabla_x\cdot \bar\jmath_t^\gamma &= \Delta_x\bar\rho_t^\gamma\,,\\
		\bar\jmath_t^\gamma(dx) &= \intr \sfF(x-v/\gamma,\rho_t^\gamma)\,(\Gamma^\gamma_\#\mu_t^\gamma)(dxdv)\,,
	\end{aligned}\right.
\end{align}
with initial condition $\bar\rho_0^\gamma=(\pi^x\circ \Gamma^\gamma)_\#\mu_0^\gamma$.

\begin{remark}
Due to the non-increasing nature of relative entropy under push-forward and the joint convexity of the functional $(\mu,\nu)\mapsto\Ent(\mu|\nu)$, we can apply Jensen's inequality to obtain the estimate
\[
	\Ent(\bar\rho^\gamma|\scrN^d) \le \Ent(\Gamma^\gamma_\#\mu^\gamma|\scrN^{2d}) = \Ent(\Gamma^\gamma_\#\mu^\gamma|(\Gamma^\gamma\circ\Gamma^{-\gamma})_\#\scrN^{2d}) \le \Ent(\mu^\gamma|\Gamma^{-\gamma}_\#\scrN^{2d})\,,
\]
which, due to Lemma \ref{lem:g-estimate}, yields
\[
	\Ent(\bar\rho_t^\gamma|\scrN^d) \le M(\mu_0^\gamma,t)\qquad\text{for all $t\in[0,T]$ and $\gamma\ge 1$}
\]
with $M(\mu_0^\gamma,t)$ as given in Lemma \ref{lem:g-estimate}.
\end{remark}

\begin{remark}
In a similar fashion as in the proof of Lemma~\ref{lem:g-estimate} (cf.\ Theorem~\ref{thm:entropy-linear}), one can establish the relative entropy estimate
	\[
		\Ent(\bar\rho_t^\gamma|\scrN^d) + \frac{1}{2}\int_0^t \intr |\nabla \log \bar\rho_r^\gamma|^2\,d\bar\rho_r^\gamma\,dr \le \Ent(\bar\rho_0^\gamma|\scrN^{2d})e^{\bar\lambda t} + (\bar\beta/\bar\lambda)(e^{\bar\lambda t}-1)
	\]
	for appropriate constants $\bar\beta,\bar\lambda>0$ independent of $\gamma\ge 1$.
\end{remark}

In the following two lemmas, we provide the quantitative bound on the error between $\bar\rho^\gamma$ and $\rho^\gamma$ in $2$-Wasserstein distance and regularity estimates on $\bar\rho^\gamma$ and $\rho^\gamma$.

\begin{lemma}[Relationship between $\bar\rho^\gamma$ and $\rho^\gamma$]\label{lem_rho2}

	Let $(\bar\rho_t^\gamma)$ and $(\rho_t^\gamma)$ be given as above. 
	Then the following estimate holds
	\begin{align*}\label{eq:rho-comparison}
		\d_2^2(\bar\rho_t^\gamma,\rho_t^\gamma) \le \frac{1}{\gamma^2}\intrr |v|^2 \,\mu_t^\gamma(dxdv)\,.
	\end{align*}
\end{lemma}
\begin{proof}
	We begin by relating $\Gamma^\gamma_\#\mu_t^\gamma$ with $\mu_t^\gamma$ for each $t\ge 0$. By choosing a coupling of the form $(\Gamma^\gamma \times id)_\#\mu_t^\gamma$ we find that
\begin{align*}
	\d_2^2(\Gamma^\gamma_\#\mu_t^\gamma,\mu_t^\gamma) &\le \intrr |\Gamma^\gamma(x,v) - (x,v)|^2\, \mu_t^\gamma(dxdv) = \frac{1}{\gamma^2}\intrr |v|^2\, \mu_t^\gamma(dxdv)\,.
\end{align*}
Now let $\Pi_t$ be an optimal coupling between $\Gamma^\gamma_\#\mu_t^\gamma$ and $\mu_t^\gamma$ for each $t\ge 0$. Then $(\pi^x\times\pi^x)_\#\Pi_t$ is a coupling between $\bar\rho_t^\gamma$ and $\rho_t^\gamma$. Thus we get
\begin{align*}
	\d_2^2(\bar\rho_t^\gamma,\rho_t^\gamma) &\le \intrr |x-\hat x|^2 ((\pi^x\times\pi^x)_\#\Pi_t)(dxd\hat x) \\
	&\le \iiiint_{\R^d\times\R^d\times\R^d\times\R^d}\Bigl(|x-\hat x|^2 + |v-\hat v|^2\Bigr)\, \Pi_t(dxdvd\hat xd\hat v) = \d_2^2(\Gamma^\gamma_\#\mu_t^\gamma,\mu_t^\gamma)\,,
\end{align*}
which, together with the previous estimate yields the assertion.
\end{proof}

\begin{lemma}[Implied regularity]\label{lem:lp-rho-u}
Let $k\in\{0,1\}$. If $\mu^\gamma\in W_{x,H}^{k,p}$, then $\rho^\gamma,\bar\rho^\gamma\in W^{k,p}(\R^d)$. In particular, if $p>d$, we have, due to Morrey's inequality, $\rho^\gamma,\bar\rho^\gamma\in L^\infty(\R^d)$. 
\end{lemma}
\begin{proof}
We only prove the statement for $\bar\rho^\gamma$---the argument for $\rho^\gamma$ holds analogously.

For any $\varphi\in\calC_c^\infty(\R^d)$ with $\|\varphi\|_{L^{p'}} \leq 1$, where $p'$ is the H\"older conjugate of $p$, we find
\begin{align*}
\intr \varphi(x)\,\bar\rho^\gamma(dx)
		&= \intrr \varphi(x) \,\Gamma^\gamma_\#\mu^\gamma(dxdv)
		= \intrr \varphi(x+v/\gamma)\,\mu^\gamma(dxdv)\,.
\end{align*}
We then bound the right-hand side as
$$\begin{aligned}
\intrr \varphi(x+v/\gamma)\,\mu^\gamma(dxdv) &\le \lt|\intrr \varphi(x+v/\gamma)e^{-\frac1{p'}H(x,v)}\,\mu^\gamma(x,v)e^{\frac1{p'}H(x,v)}\,\leb^{2d}(dxdv)\rt| \\
&\le \left(\intrr |\varphi(x+v/\gamma)|^{p'} e^{-H(x,v)}\leb^{2d}(dxdv)\right)^{1/p'} \|\mu^\gamma\|_{L_{H}^p}\,.
\end{aligned}$$
Recall that $H(x,v) = \Phi(x) + |v|^2/2$ and $\Phi \geq 0$. Now since
\begin{align*}
\intrr |\varphi(x+v/\gamma)|^{p'} e^{-H(x,v)}\leb^{2d}(dxdv) &\le (2\pi)^{d/2}\intrr |\varphi(x+v/\gamma)|^{p'}\leb^d(dx)\scrN^{d}(dv)\\
 &= (2\pi)^{d/2}\|\varphi\|_{L^{p'}}^{p'},
\end{align*}
we deduce, by duality, that $\bar\rho^\gamma\in L^p(\R^d)$.
%


Now let $\varphi\in\calC_c^\infty(\R^d)$ with $\|\varphi\|_{L^{p'}}\le 1$ be arbitrary. Then elementary computations yield
\begin{align*}
\lt|\intr (\partial_{x_i}\varphi)(x)\,\bar\rho^\gamma(dx)\rt| &= \lt|\intrr (\partial_{x_i}\varphi)(x)\,(\Gamma^\gamma_\#\mu^\gamma)(dxdv)\rt| \cr
&= \lt|\intrr (\partial_{x_i}\varphi)(x+v/\gamma)\,\mu^\gamma(x,v)\,\leb^{2d}(dxdv)\rt| \\
&= \lt|\intrr \varphi(x+v/\gamma)\,(\partial_{x_i}\mu^\gamma)(x,v)\,\leb^{2d}(dxdv)\rt| \\
&\le \left(\intrr |\varphi(x+v/\gamma)|^{p'} e^{-H(x,v)}\leb^{2d}(dxdv)\right)^{1/p'} \|\partial_{x_i}\mu^\gamma\|_{L_{H}^p}\,,
\end{align*}
for each $i=1,\ldots,d$. 
As above, we deduce $\nabla\bar\rho^\gamma\in L^p(\R^d)$, and consequently, $\bar\rho^\gamma\in W^{1,p}(\R^d)$.
\end{proof}

%
%
%
%
%
\section{Modulated interaction energy and Wasserstein stability estimates}\label{sec:2WEV}

This section contains two  essential parts to our strategy: The first part concerns the so-called {\em modulated interaction energy} between two measures $\mu$, $\nu$ for a given interaction kernel $K$. This energy appears when considering the differences between the free energies $\scrE(\mu)$ and $\scrE(\nu)$, which needs to be controlled in terms of the 2-Wasserstein distance $\d_2(\mu,\nu)$. The second part of this section discusses the temporal derivative of $\d_2^2(\mu_t,\nu)$ along $\scrE$-regular solutions $t\mapsto\mu_t$ of \eqref{eq:diffusion} and an arbitrary measure $\nu$. The temporal derivate can be related to differences in free energies, and ultimately to $\d_2^2(\mu_t,\nu)$.

\subsection{Modulated interaction energy}\label{sec:modulated}

For measures $\mu,\nu\in\calP(\R^d)$ we consider the modulated interaction energy
\[
	\scrD_K(\mu,\nu):=\intrr K(x-y)(\mu-\nu)(dy)(\mu-\nu)(dx)\,,
\]
and further recall the $\dot H^{-1}$ homogeneous Sobolev norm given by
\[
	\|\omega\|_{\dot H^{-1}} :=\sup\Bigl\{ |\langle f,\omega\rangle| \,:\, \|\nabla f\|_{L^2} \le 1\Bigr\}
\]
for any signed measure $\omega$ on $\R^d$. We remark that this norm is only finite for signed measures having zero total mass, i.e.\ $\omega(\R^d)=0$.

\begin{remark}\label{rem:DK-homogeneous}
	When $\nabla K\star(\mu-\nu)\in L^2(\R^d)$, then $\scrD_K$ can be related to $\|\cdot\|_{\dot{H}^{-1}}$ by
\[
	|\scrD_K(\mu,\nu)| \le \|\nabla K\star(\mu-\nu)\|_{L^2(\R^d)}\|\mu-\nu\|_{\dot{H}^{-1}}\,.
\]	
\end{remark}

The next proposition relates the homogeneous $\dot H^{-1}(\R^d)$ distance and the 2-Wasserstein distance for measures that have essentially bounded Lebesgue densities \cite{Loe06,Pey18}.

\begin{proposition}\label{prop:W2-homogeneous}
	Let $\mu,\nu\in \calP_2(\R^d)$ with 
	\[
		c_\infty:=\max\bigl\{\|\mu\|_{L^\infty},\|\nu\|_{L^\infty}\bigr\} <\infty\,,
	\]
	then
	\[
		\|\mu-\nu\|_{\dot H^{-1}} \le \sqrt{c_\infty}\,\d_2(\mu,\nu)\,.
	\]
\end{proposition}

Due to Remark \ref{rem:DK-homogeneous} and Proposition \ref{prop:W2-homogeneous}, we can relate $\scrD_K$ with $\d_2$. However, as we will observe in the next sections, when $\scrD_K$ take negative values, we will require $|\scrD_K|$ to be controlled by $\d_2^2$. This can be achieved for appropriate interaction kernels as given in the next statement.

\begin{theorem}\label{thm:DK-interaction}
\begin{enumerate}
\item Smooth interaction: If $\nabla K$ is globally Lipschitz, then
\[
	|\scrD_K(\mu,\nu)| \le \|\nabla K\|_{\Lip}\, \d_2^2(\mu,\nu)\,.		
\]
		
		\item Weakly singular interaction: If $\nabla^2 K\in L^1(\R^d)$, then
		\[
			\|\nabla K\star(\mu-\nu)\|_{L^2} \le \|\nabla^2 K\|_{L^1}\,\d_2(\mu,\nu)\,.
		\]
		In particular, if $c_\infty<\infty$, with $c_\infty$ given in Proposition \ref{prop:W2-homogeneous}, then
		\[
			|\scrD_K(\mu,\nu)| \le \sqrt{c_\infty}\,\|\nabla^2 K\|_{L^1}\,\d_2^2(\mu,\nu)\,.
		\]
	
	\item Newtonian attractive: When $K$ is the fundamental solution of the Laplacian, i.e.\ $\Delta K = \delta_0$, $\scrD_K$ takes the alternative form
	\[
		\scrD_K(\mu,\nu) = -\intr |\nabla K\star(\mu-\nu)|^2\,d\leb^d=-\|\mu-\nu\|_{\dot{H}^{-1}}^2\,,
	\]
	from which we obtain
	\[
		\scrD_K(\mu,\nu) \ge -c_\infty \d_2^2(\mu,\nu)\,.
	\]
	\end{enumerate}
\end{theorem}
\begin{proof}
	\begin{enumerate}
		\item We begin by noticing that 
\begin{align}\label{smooth}
\begin{aligned}
	|\scrD_K(\mu,\nu)| &= \lt|\intr (K \star (\mu - \nu))(x) (\mu-\nu)(dx)\rt| \cr
& = \lt|\intrr \bigl[ (K \star (\mu - \nu))(x)  - (K \star (\mu - \nu))(\hat x) \bigr] \pi(dx d\hat x)\rt| \cr
&\leq \|\nabla K \star (\mu - \nu)\|_{L^\infty} \intrr |x - \hat x| \,\pi(dx d\hat x), 	
\end{aligned}
\end{align}
where $\pi$ is an arbitrary coupling between $\mu$ and $\nu$. Optimizing over all couplings of $\mu$ and $\nu$ in \eqref{smooth} yields
\bq\label{sm2}
	|\scrD_K(\mu,\nu)| \leq \|\nabla K \star (\mu - \nu)\|_{L^\infty} \d_{1}(\mu,\nu).
\eq
Similarly, we  find 
$$\begin{aligned}
\lt|\intr \nabla K(x-y)(\mu-\nu)(dy)\rt| &= \lt|\intrr \lt(\nabla K(x - y) - \nabla K(x - \hat y)\rt) \pi (dy d\hat y)\rt|\cr
&\leq \|\nabla K \|_{\Lip}\intrr |y - \hat y| \, \pi (dy d\hat y),
\end{aligned}$$
and thus, for all $x\in\R^d$,
\[
\lt|\intr \nabla K(x-y)(\mu-\nu)(dy)\rt| \leq  \|\nabla K \|_{\Lip} \d_1(\mu,\nu).
\]
This together with \eqref{sm2} implies
\[
\lt|\scrD_K(\mu,\nu)\rt| \leq   \|\nabla K \|_{\Lip}\d_1^2(\mu,\nu) \leq   \|\nabla K \|_{\Lip}\d_2^2(\mu,\nu)
\]
due to the monotonicity in $p$ of the $p$-Wasserstein distances.

\item An application of Remark~\ref{rem:DK-homogeneous} and Proposition~\ref{prop:W2-homogeneous} yields
\bq\label{dk1}
	|\scrD_K(\mu,\nu)| \leq \sqrt{c_\infty}\,\d_2(\mu,\nu)\|\nabla K\star(\mu - \nu)\|_{L^2}.
\eq
Thus it remains to show that 
\[
\|\nabla K\star(\mu - \nu)\|_{L^2} \leq C\d_2(\mu,\nu)
\]
for some constant $C>0$. Notice that
\begin{align*}
\intr K(x-y)(\mu-\nu)(dy) &= \intr \bigl(K(x-y)-K(x-\sfT(y))\bigr)\,\mu(dy) \\
&=\int_0^1 \intr (\nabla K)(x-\sfT^\theta(y))\cdot(y-\sfT(y))\,\mu(dy)\,d\theta\,,
\end{align*}
where $\sfT$ is the \textit{optimal} transport map such that $\sfT_\#\mu=\nu$
and $\sfT^\theta(y)=(1-\theta)\sfT(y)+\theta y$. 

Taking gradient of the above gives
\[
\nabla\intr K(x-y)(\mu-\nu)(dy) = \int_0^1 \intr (\nabla^2 K)(x-\sfT^\theta(y))(y-\sfT(y))\,\mu(dy)\,d\theta\,.
\]
By an application of H\"older's inequality, the inner integral of the right-hand side above can be estimated as
\begin{align*}
&\intr (\nabla^2 K)(x-\sfT^\theta(y))(y - \sfT(y))\,\mu(dy) \\
&\qquad \le \left(\intr |\nabla^2 K|(x-\sfT^\theta(y))|y - \sfT(y)|^2\,\mu(dy)\right)^{1/2} \left(\intr |\nabla^2 K|(x-\sfT^\theta(y))\,\mu(dy)\right)^{1/2}.
\end{align*}
From the properties of displacement interpolants $\sfT^\theta_\#\mu$ \cite[Theorem 2.3]{Mc97}, we deduce that $\sfT^\theta_\#\mu\ll \leb^d$ for all $\lambda\in(0,1)$ and that 
\[
\|\sfT^\theta_\#\mu\|_{L^\infty} \le\max\bigl\{\|\mu\|_{L^\infty},\|\nu\|_{L^\infty}\bigr\} =  c_\infty\,.
\] 
Consequently, the second term on the right-hand side of the former equation can be further estimated
\[
\intr |\nabla^2 K|(x-\sfT^\theta(y))\,\mu(dy) = \intr |\nabla^2 K|(x-y)\,(\sfT^\theta_\#\mu)(dy) 
\le c_\infty\, \|\nabla^2 K\|_{L^1}\,.
\]
Hence, by using Fubini's theorem, we obtain
\begin{align*}
&\|\nabla K\star(\mu - \nu)\|_{L^2}^2\cr
&\quad  \le c_\infty \|\nabla^2 K\|_{L^1} \int_0^1\intrr (\nabla^2 K)(x-\sfT^\theta(y))|y - \sfT(y)|^2\,\mu(dy)\,\leb^{d}(dx)\,d\theta \\
&\quad = c_\infty \|\nabla^2 K\|_{L^1}^2 \intr |y-\sfT(y)|^2\,\mu(dy) = c_\infty \|\nabla^2 K\|_{L^1}^2 \d_2^2(\mu,\nu)\,.
\end{align*}
We finally combine this with \eqref{dk1} to conclude the desired inequality.

	\item We first notice that $\scrD_K$ can be rewritten as 
\bq\label{d_k1}
\scrD_K(\mu,\nu) = -\intr |\nabla K\star(\mu-\nu)(x)|^2\,d\leb^d\,.
\eq
Indeed, we get
\[
\scrD_K(\mu,\nu) = \intr (K \star (\mu - \nu))(x) (\Delta K \star (\mu - \nu))(x)\,d\leb^d\,,
\]
and applying the integration by parts to the right hand side yields \eqref{d_k1}. We next show that 
\bq\label{hl_2}
\|\mu-\nu\|_{\dot{H}^{-1}} = \|\nabla K \star (\mu - \nu)\|_{L^2}\,.
\eq
For any $\varphi \in \dot{H}^1(\R^d)$ with $\|\varphi\|_{\dot{H}^1} \leq 1$, we find
\begin{align*}
\lt|\intr \varphi(x)(\mu - \nu)(dx) \rt| &= \lt|\intr \varphi(x) (\Delta K \star (\mu - \nu))(x)\,d\leb^d \rt|\cr
&= \lt|\intr \nabla \varphi(x) \cdot (\nabla K \star (\mu - \nu))(x)\,d\leb^d \rt|,
\end{align*}
and hence,
\[
	\|\mu-\nu\|_{\dot{H}^{-1}} = \sup_{\|\nabla \varphi\|_{L^2}\le 1} \lt|\intr \nabla \varphi(x) \cdot (\nabla K \star (\mu - \nu))(x)\,d\leb^d \rt| = \|\nabla K\star (\mu-\nu)\|_{L^2}\,.
\]
This gives the assertion \eqref{hl_2}. Hence we have
\[
\scrD_K(\mu,\nu) = -\|\nabla K \star (\mu - \nu)\|_{L^2}^2 = - \|\mu-\nu\|_{\dot{H}^{-1}}^2 \ge - c_\infty \d_2^2(\mu,\nu)\,,
\]
where the last inequality follows from Proposition \ref{prop:W2-homogeneous}.
\end{enumerate}
\end{proof}

\begin{remark}
	\begin{enumerate}[label=(\roman*)]
		\item For purely repulsive interactions, i.e.\ interaction potentials $K$ giving rise to positive definite kernels, and hence $\scrD_K(\mu,\nu)\ge 0$ for any $\mu,\nu\in\calP(\R^d)$. 
		\item As mentioned in the introduction, it is possible to consider linear combinations of interaction potentials satisfying each of the conditions above. For instance, one could consider an interaction kernels satisfying
	\[
		\|K\|_{W^{2,1}(B_{2R})} + \|\nabla^2 K\|_{L^\infty(\R^d\setminus B_{R})}<\infty\,
	\]
	for some $R>0$, we can combine Theorem~\ref{thm:DK-interaction}(1) and \ref{thm:DK-interaction}(2) to obtain 
	\[
		|\scrD_K(\mu,\nu)| \le c_{K,R,\infty}\,\d_2^2(\mu,\nu)\,,
	\]
	for some constant $c_{K,R,\infty}>0$, depending only on $K$, $R$, and $c_\infty$.

	\end{enumerate}	
\end{remark}

\subsection{Stability estimates for 2-Wasserstein gradient flows}\label{sec:EVI}

As mentioned in the introduction, the aggregation-diffusion equation \eqref{eq:diffusion} may be formulated as a 2-Wasserstein (or Otto) gradient flow.
Under sufficient regularity of the solution to \eqref{eq:diffusion}, this formulation may be made rigorous \cite{AGS05,Vil03}. 
%
More precisely, $\scrE$-regular solutions (cf.\ Definition~\ref{def:regular-solution}) to \eqref{eq:diffusion} may be 
characterized by the family of {\em Evolution Variational Inequalities (EVI)}
\begin{align}\label{eq:EVI}
	\frac{1}{2}\frac{d}{dt}\d_2^2(\rho_t,\nu) \le \scrE(\nu) - \scrE(\rho_t) + \frac{\lambda}{2}\,\d_2^2(\rho_t,\nu)\qquad\text{for all $\nu\in\text{dom}(\scrE)$}\,,
\end{align}
for an appropriate constant $\lambda\in\R$.

In the case when $\lambda\in\R$ is independent of $\rho$ and $\nu$, the EVI may be used to deduce the uniqueness of gradient flow solutions. This follows from \cite[Lemma 4.3.4]{AGS05} (cf.\ \cite[Theorem 11.1.4]{AGS05} for the uniqueness argument; see also \cite{CLM2014}).
\begin{align}\label{eq:W2-doubling}
	\frac{1}{2}\frac{d}{dt}\d_2^2(\rho_t^1,\rho_t^2) \le \frac{1}{2}\frac{d}{dt}\d_2^2(\rho_t^1,\rho_s^2)\Bigl|_{s=t} \ +\ \frac{1}{2}\frac{d}{ds}\d_2^2(\rho_t^1,\rho_s^2)\Bigl|_{t=s} \le \lambda\,\d_2^2(\rho_t^1,\rho_t^2)\,,
\end{align}
which, by application of the Gr\"onwall's lemma, yields the stability estimate
\[
	\d_2(\rho_t^1,\rho_t^2) \le e^{\lambda t}\,\d_2(\rho_0^1,\rho_0^2)\,\qquad\text{for all $t\ge 0$}\,.
\]

\medskip

In the following, we establish a similar system of inequalities for the solution to our intermediate system \eqref{eq:aux-rho}, which we recall here for convenience:
\begin{align*}
	\left\{\quad
	\begin{aligned}
		\partial_t\bar\rho_t^\gamma + \nabla\cdot \bar\jmath_t^\gamma &= \Delta\bar\rho_t^\gamma\,,\\
		\bar\jmath_t^\gamma(dx) &= \intr \sfF(x-v/\gamma,\rho_t^\gamma)\,(\Gamma^\gamma_\#\mu_t^\gamma)(dxdv)\,.
	\end{aligned}\right.
\end{align*}
We begin by recalling a known result for the temporal derivative of the 2-Wasserstein distance along the continuity equation.

\begin{proposition}[\cite{AGS05,Vil03}]\label{prop:W2-derivative}
	Let $(\varrho_t)_{t\in[0,T]}\subset(\calP_2\cap L^1)(\R^d)$ satisfy the continuity equation
	\[
		\partial_t\varrho_t + \nabla\cdot(\varrho_t\, \xi_t) = 0	\qquad\text{in the sense of distributions}
	\]
	for some vector field with $\int_0^T\|\xi_t\|_{L^2(\varrho_t)}^2\,dt<\infty$. Then for any $\nu\in(\calP_2\cap L^1)(\R^d)$,
	\[
		\frac{1}{2}\frac{d}{dt}\d_2^2(\varrho_t,\nu) = \intr \langle x-\sfT_t(x), \xi_t(x)\rangle\, d\varrho_t\,,
	\]
	where $\sfT_t{}_\#\varrho_t=\nu$ for almost every $t\in(0,T)$.
\end{proposition}

To deduce the EVI for the evolutions we consider, we will need to compute the subdifferential of $\scrE$ with respect to the 2-Wasserstein distance. While the following proposition is simply a combination of results known in the literature, we include the proof for the sake of completeness.

\begin{proposition}\label{prop:subdifferential}
	Let $\rho,\nu\in \calP_2(\R^d)\cap\text{dom}(\scrE)$ with $\scrD_K(\rho,\nu)<\infty$ and 
	\[
		\intr \frac{|\nabla\rho|^2}{\rho}\, d\leb^d <\infty,\qquad \nabla K\star\rho\in L^\infty(\R^d)\,,
	\]
	and let $\sfT:\R^d\to\R^d$ be the optimal transport map between $\rho$ and $\nu$, i.e.\ $\sfT{}_\#\rho=\nu$.  Then
		\begin{align}\label{eq:subdifferential}
		\begin{aligned}
		\scrE(\nu) - \scrE(\rho) &\ge \intr \langle \sfT(x)-x,(\nabla\rho)(x)\rangle\,d\leb^d + \frac{1}{2}\scrD_K(\rho,\nu) \cr
		&\qquad\qquad - \int_0^1 \intr \langle \sfT(x)-x,\sfF(\sfT^\theta(x),\rho) \rangle\,\rho(dx)\,d\theta\,,
		\end{aligned}
	\end{align}
	where $\sfT^\theta(x) = (1-\theta)x + \theta \sfT(x)$, $\theta\in(0,1)$ for $\rho$-almost every $x\in\R^d$.
\end{proposition}
\begin{proof}
Due to the convexity of $\theta \mapsto \Ent(\sfT^\theta_\#\rho|\leb^d)$ \cite{Mc97}, we have that
\begin{align*}
	\Ent(\nu|\leb^d) - \Ent(\rho|\leb^d) &\ge \frac{1}{\theta}\lt(\Ent(\sfT^\theta_\#\rho|\leb^d) - \Ent(\rho|\leb^d)\rt),\qquad\theta\in(0,1)\,.
\end{align*}
In particular, the inequality holds true also in the limit $\theta\to 0$, which yields 
\begin{align*}
	\Ent(\nu|\leb^d) - \Ent(\rho|\leb^d) 
	&\ge -\intr \rho(x)\,\text{tr}(\tilde\nabla(\sfT-id))(x)\,d\leb^d\,,
\end{align*}
(cf.\ \cite[Lemma 10.4.4]{AGS05}) where $\text{tr}\tilde\nabla$ denotes the approximate divergence. 

By means of approximation with smooth cut-off functions and the Calderon--Zygmund theorem (cf.\ \cite[Theorem~3.83]{AFP2000}), one can show that the (weak) integration by parts hold under the above assumption on $\rho$, i.e.\ we have that
\[
	-\intr \rho(x)\,\text{tr}(\tilde\nabla(\sfT-id))(x)\,d\leb^d \ge \intr \langle \sfT(x) - x,(\nabla\rho)(x)\rangle\,d\leb^d\,.
\]
Together, this yields the inequality
\begin{align*}
	\Ent(\nu|\leb^d) - \Ent(\rho|\leb^d) \ge \intr \langle \sfT(x)-x,\nabla\rho(x)\rangle\,\leb^{d}(dx)\,.
\end{align*} 
As for the other terms, we simply perform basic algebraic manipulations to obtain
\begin{align*}
	&\intr (K\star \nu)(x)\,\nu(dx) - \intr (K\star \rho)(x)\,\rho(dx) \\
	&\quad\qquad = \intr (K\star (\nu-\rho))\, \nu(dx) + \intr (K\star\rho)(x)\,\nu(dx) - \intr (K\star \rho)(x)\,\rho(dx) \\
	&\quad\qquad = \scrD_K(\rho,\nu) + 2\intr \lt[(K\star\rho)(\sfT(x)) - (K\star\rho)(x)\rt] \rho(dx) \\
	&\quad\qquad = \scrD_K(\rho,\nu) + 2\int_0^1 \intr \langle \sfT(x)-x,(\nabla K\star\rho)(\sfT^\theta(x))\rangle\, \rho(dx)\,d\theta\,,
\end{align*}
and 
\begin{align*}
	\intr \Phi(x)\,\nu(dx) - \intr \Phi(x)\,\rho(dx) = \int_0^1\intr \langle \sfT(x)-x,(\nabla\Phi)(\sfT^\theta(x))\rangle\,\rho(dx)\,d\theta\,.
\end{align*}
Summing each of the contributions concludes the proof.
\end{proof}

%
%
%
%

By combining Propositions~\ref{prop:W2-derivative} and \ref{prop:subdifferential}, we easily obtain the following result, which ultimately allows us to deduce our main stability estimate \eqref{EVI0}. 

\begin{theorem}\label{thm:EVI-like}
	Let $(\varrho_t)_{t\in[0,T]}\subset\calP(\R^d)\cap\text{dom}(\scrE)$ satisfy
	\[
		\partial_t\varrho_t + \nabla\cdot\bigl(\varrho_t\,(\sfF(\cdot,\varrho_t) + \mathsf{e}_t)\bigr) = \Delta\varrho_t\quad \mbox{with} \quad \int_0^T\intr\frac{|\nabla \varrho_t|^2}{\varrho_t} \,d\leb^ddt<\infty\,,
	\]
	where 	
	$\mathsf{e}_t$ is a vector field with $\int_0^T \|\mathsf{e}_t\|_{L^2(\varrho_t)}^2\,dt <\infty$ and $\sfF(\cdot,\varrho)=-\nabla\Phi-\nabla K\star\varrho$. Suppose further that the map $x\mapsto\sfF(x,\varrho_t)$ is Lipschitz for every $t\in[0,T]$ with
	\[
		c_0:=\sup_{t\in[0,T]}\|\sfF(\cdot,\varrho_t)\|_{\Lip}<\infty\,.
	\]
	Then for any $\nu\in\text{dom}(\scrE)$, the evolution-variational-like inequality
	\begin{align}\label{eq:EVI-like}
		\frac{1}{2}\frac{d}{dt}\d_2^2(\varrho_t,\nu) &\le \scrE(\nu) - \scrE(\varrho_t) - \frac{1}{2}\scrD_K(\varrho_t,\nu) + \frac{\lambda}{2}\,\d_2^2(\varrho_t,\nu) + \frac{1}{2}\|\mathsf{e}_t\|_{L^2(\varrho_t)}^2
	\end{align}
	holds for almost every $t\in(0,T)$, with $\lambda = 1 + c_0$.
	
	In particular, if $\mathsf{e}\equiv 0$, then we simply obtain the evolution variational inequality \eqref{eq:EVI}.
\end{theorem}
\begin{proof}
	By the integrability assumption on the velocity fields, we have that $t\mapsto\varrho_t$ is absolutely continuous with respect to  $\d_2$ and thus Proposition \ref{prop:W2-derivative} applies. Together with \eqref{eq:subdifferential}, we obtain
\begin{align*}
\frac{1}{2}\frac{d}{dt}\d_2^2(\varrho_t,\nu) 
&\le \scrE(\nu) - \scrE(\varrho_t) - \frac{1}{2}\scrD_K(\varrho_t,\nu) + \intr \langle x-\sfT_t(x), \mathsf{e}_t(x)\rangle\, d\varrho_t \\
&\hspace{6em} - \int_0^1 \intr \langle \sfT_t(x)-x,\sfF(x,\varrho_t) - \sfF(\sfT_t^\theta(x),\varrho_t)\rangle\,d\varrho_t\,d\theta \\
&\le \scrE(\nu) - \scrE(\varrho_t) - \frac{1}{2}\scrD_K(\varrho_t,\nu) + \frac{1}{2}\,\d_2^2(\varrho_t,\nu) + \frac{1}{2}\|\mathsf{e}_t\|_{L^2(\varrho_t)}^2 \\
&\hspace{6em} + c_0 \d_2(\varrho_t,\nu)\int_0^1 \d_2(\sfT_t^\theta {}_\#\varrho_t,\nu)\,d\theta\,,
\end{align*}
	where $\sfT^\theta_t(x) = (1-\theta)x + \theta \sfT_t(x)$, $\theta\in(0,1)$ for $\varrho_t$-almost every $x\in\R^d$. From the property of geodesic interpolants \cite{Mc97}, we have that $\d_2(\sfT_t^\theta{}_\#\varrho_t,\nu)=\theta \d_2(\varrho_t,\nu)$, and hence,
	\[
		\int_0^1 \d_2(\sfT_t^\theta {}_\# \varrho_t,\nu)\,d\theta = \frac{1}{2}\d_2(\varrho_t,\nu)\,,
	\]
	thereby concluding the proof.
\end{proof}

Due to Theorem \ref{thm:EVI-like}, we find that $(\bar\rho_t^\gamma)$, solving the intermediate problem \eqref{eq:aux-rho}, satisfies the evolution-variational-like inequality \eqref{eq:EVI-like} with
\[
	\mathsf{e}_t^\gamma(x) := \frac{d\bar\jmath_t^\gamma}{d\bar\rho_t^\gamma}(x) - \sfF(x,\bar\rho_t^\gamma)\qquad\text{for $\bar\rho_t^\gamma$-almost every $x\in\R^d$\,.}
\]
Consequently, we deduce our main stability estimate from the first equality in \eqref{eq:W2-doubling}: 
\begin{align}\label{eq:W2-stability}
	\frac{1}{2}\frac{d}{dt}\d_2^2(\bar\rho_t^\gamma,\rho_t) \le  \lambda \d_2^2(\bar\rho_t^\gamma,\rho_t) - \scrD_K(\bar\rho_t^\gamma,\rho_t) + \frac{1}{2}\|\mathsf{e}_t^\gamma\|_{L^2(\bar\rho_t^\gamma)}^2\,.
\end{align}

\begin{remark}
	In addition to requiring $\int_0^T \|\mathsf{e}_t^\gamma\|_{L^2(\bar\rho_t)}^2\,dt <\infty$, $(\bar\rho_t^\gamma)$ is required to satisfy the same regularity assumptions for $\scrE$-regular solutions (cf.\ Definition~\ref{def:regular-solution}) for Theorem~\ref{thm:EVI-like} to apply.
\end{remark}


Finally, we provide the integrability estimate of $\mathsf{e}_t^\gamma$.
\begin{lemma}\label{lem_ee} The following holds.
\[
\|\mathsf{e}_t^\gamma\|_{L^2(\bar\rho_t^\gamma)}^2 \leq \intrr \lt|\sfF(x,\rho_t^\gamma) - \sfF(x+v/\gamma,\bar\rho_t^\gamma)\rt|^2 d\mu_t^\gamma.
\]
\end{lemma}
\begin{proof}
For any $\varphi \in L^2(\bar\rho_t^\gamma)$, we have
\begin{align*}
\lt|\intr \langle \varphi(x),\mathsf{e}_t^\gamma(x)\rangle\,d\bar\rho_t^\gamma\rt| &= \lt|\intrr \langle \varphi(x), \sfF(x-v/\gamma,\rho_t^\gamma) - \sfF(x,\bar\rho_t^\gamma)\rangle\,d\Gamma^\gamma_\#\mu_t^\gamma\rt| \\
&\le \|\varphi\|_{L^2(\bar\rho_t^\gamma)}\lt(\intrr \lt|\sfF(x,\rho_t^\gamma) - \sfF(x+v/\gamma,\bar\rho_t^\gamma)\rt|^2 d\mu_t^\gamma\rt)^{1/2},
\end{align*}
due to the fact that $\|\varphi\|_{L^2(\Gamma^\gamma_\#\mu_t^\gamma)} = \|\varphi\|_{L^2(\bar\rho_t^\gamma)}$. This completes the proof.
\end{proof}

%
%
%
%
%
%
%
\section{Stability for bounded and Lipschitz interaction forces}\label{sec:bL}

We finally arrive at our first quantitative estimate that applies to globally Lipschitz continuous and bounded interaction forces, i.e.,
\[
		c_K:=\|\nabla K\|_{L^\infty} + \|\nabla K\|_{\Lip} <\infty\,.
\]
Under this assumption, one easily verifies that
\begin{subequations}
\begin{enumerate}[label=(\roman*)]
	\item for any $\mu\in\calP(\R^d)$,
	\begin{align}\label{eq:smooth-property-1}
		\|\nabla K\star\mu\|_{L^\infty} \le \|\nabla K\|_{L^\infty} \leq c_K\,,
	\end{align}
	\item and for any $\mu,\nu\in\calP_2(\R^d)$ and every $x, y\in\R^d$,
	\begin{align}\label{eq:smooth-property-2}
		|\sfF(x,\mu) - \sfF(y,\nu)| \le c_{\sfF}\lt(|x- y| + \d_2(\mu,\nu)\rt)
	\end{align}
	with $c_{\sfF} = \|\nabla \Phi\|_{\Lip} + \|\nabla K\|_{\Lip}$. Indeed, we find 
	\[
	|\sfF(x,\mu) - \sfF(y,\nu)| \leq |(\nabla \Phi)(x) - (\nabla \Phi)(y)| + |(\nabla K\star\mu)(x) - (\nabla K\star\nu)(y)|,
	\]
	where we can easily obtain $ |(\nabla \Phi)(x) - (\nabla \Phi)(y)| \leq  \|\nabla \Phi\|_{\Lip}|x-y|$. On the other hand, if we set $\sfT$ as an optimal map giving $\nu = \sfT_\# \mu$, then
	\begin{align*}
	||(\nabla K\star\mu)(x) - (\nabla K\star\nu)(y)| &\leq \intr |(\nabla K) (x-z) - (\nabla K)(y-\sfT(z))| \,\mu(dz)\cr
	&\leq \|\nabla K\|_{\Lip} \intr (|x-y| + |z - \sfT(z)|)\, \mu(dz)\cr
	&\leq \|\nabla K\|_{\Lip}\lt(|x-y| + \d_2(\mu,\nu) \rt).
	\end{align*}
	Combining the above estimates yields \eqref{eq:smooth-property-2}.
\end{enumerate}
\end{subequations}




\begin{proof}[Proof of Theorem~\ref{thm_reg}]
	Due to the assumed regularity on $\nabla K$, one easily establishes the global Lipschitz continuity of $x\mapsto \sfF(x,\rho_t)$. The existence of weak solutions $\mu^\gamma\in\calC([0,T];\calP_2(\R^d\times\R^d))$ to \eqref{eq:kinetic-rescaled} can then be established by, e.g.\ standard theory of stochastic differential equations applied to the nonlinear McKean equation \eqref{mckean}. Under the assumption on the initial condition $\mu_0^\gamma$, Lemma~\ref{lem:g-estimate} together with Remark \ref{rmk:2m} applies, thereby giving the first part of the theorem.
	
	As for the second part, we begin by providing an estimate for $\|\mathsf{e}_t^\gamma\|_{L^2(\bar\rho_t^\gamma)}$. From the estimate \eqref{eq:smooth-property-2}, and Lemmas~\ref{lem_ee} and \ref{lem_rho2}, we estimate
\[
\|\mathsf{e}_t^\gamma\|_{L^2(\bar\rho_t^\gamma)}^2 \leq c_{\sfF}^2 \intrr \lt( \frac{|v|^2}{\gamma^2} + \d_2^2(\rho_t^\gamma,\bar\rho_t^\gamma)\rt) d\,\Gamma^\gamma_\#\mu_t^\gamma \leq  2\frac{c_{\sfF}^2}{\gamma^2}\intrr |v|^2 \,d\mu^\gamma_t\,.
\]
Along with Theorem~\ref{thm:DK-interaction}, the stability estimate \eqref{eq:W2-stability} yields
\[
	\frac{1}{2}\frac{d}{dt}\d_2^2(\bar\rho_t^\gamma,\rho_t) \le  (\lambda + c_K) \d_2^2(\bar\rho_t^\gamma,\rho_t) + 2\frac{c_{\sfF}^2}{\gamma^2}\intrr |v|^2 \,d\mu^\gamma_t\,.
\]
An application of Gr\"onwall's lemma then gives
\[
	\d_2^2(\bar\rho_t^\gamma,\rho_t) \le \lt(\d_2^2(\bar\rho_0^\gamma,\rho_0) + \frac{4c_{\sfF}^2}{\gamma^2}\int_0^t \intrr |v|^2 \,d\mu^\gamma_s\,ds\rt)e^{2(\lambda + c_K)t}\,.
\]
Making use of the triangle inequality for $\d_2$ and applying Lemma~\ref{lem_rho2} twice, we obtain
\begin{align*}
	\d_2^2(\rho_t^\gamma,\rho_t) &\le 2\d_2^2(\rho_t^\gamma,\bar\rho_t^\gamma) + 2\d_2^2(\bar\rho_t^\gamma,\rho_t) \\
	&\le c \lt(\d_2^2(\bar\rho_0^\gamma,\rho_0) + \frac{1}{\gamma^2}\int_0^t \intrr |v|^2 \,d\mu^\gamma_s\,ds\rt) \\
	&\le C \lt( \d_2^2(\rho_0^\gamma,\rho_0) + \frac{1}{\gamma^2}\rt)\qquad\text{for every $t\in[0,T]$\,,}
\end{align*}
where the constants $c,C>0$ are independent of $\gamma\ge 1$, thereby concluding the proof.
\end{proof}

%
%
%
%
%
%
%

\section{Uniform-in-$\gamma$ estimates in the weighted Sobolev space}\label{sec:uniform-gamma}

As stated in the introduction, in order to consider the singular interaction potentials, we need better regularity of the force field $\sfF$. This subsequently requires some uniform-in-$\gamma$ estimates of solutions in the higher-order Sobolev space compared to the regular interaction potential case. The main purpose of this section is, thus, to establish Theorem~\ref{thm:well-posedness}. 

However, since the proof is rather long and technical, we decided to only prove a priori estimates for sufficiently smooth solutions to the kinetic equation \eqref{eq:kinetic-rescaled} in this section, and postpone to complete proof of Theorem~\ref{thm:well-posedness} to Appendix~\ref{app:WP}. In particular, Theorem~\ref{thm:well-posedness} is a restatement of Theorem~\ref{thm:main-app}.

\begin{proposition}\label{lem:regularity-summary}
	Let $p\in[2,\infty)$, $k\in\{0,1\}$, and $(\mu_t^\gamma)_{t\in[0,T]}\subset\calP(\R^d\times\R^d)$, $\gamma>0$, be a sufficiently smooth weak solution of \eqref{eq:kinetic-rescaled} with
	\[
		\sup_{\gamma\ge 1}\|\mu_0^\gamma\|_{W_{x,H}^{k,p}} <\infty,\qquad H(x,v) = \Phi(x) + |v|^2/2\,.
	\]
	If
	\[
		\|\nabla K\star \varrho \|_{W^{k,\infty}} \le c_{K,p}\lt(1 + \|\nu\|_{W_{x,H}^{k,p}}\rt),\qquad \varrho = \pi^x_\#\nu,\qquad \text{$\nu\in W_{x,H}^{k,p}(\R^d\times\R^d)\,,$}
	\]
	for some constant $c_{K,p}>0$, then there exists some $T_p\in(0,T]$ such that
	\begin{align*}
		\sup\nolimits_{t\in[0,T_p]}\|\mu_t^\gamma\|_{W_{x,H}^{k,p}} \le M_k\,,
	\end{align*}
	where $M_k>0$ is a constant independent of $\gamma>0$.
	
	In particular, we have the estimate
	\[
		\sup\nolimits_{t\in[0,T_p]}\|\nabla K\star \rho_t^\gamma \|_{W^{k,\infty}} \le c_{K,p}(1 + M_k)\,,
	\]
	independently of $\gamma>0$.
\end{proposition}


To establish Proposition~\ref{lem:regularity-summary}, we first define 
\[
	\eta(dxdv):=\sigma_{\Phi}(dx)\scrN^d(dv)\,,\qquad \sigma_\Phi(dx) := e^{-\Phi(x)}\leb^d(dx)\,.
\]
Here, $\scrN^d$ represents the standard normal distribution on $\R^d$. We note that considering $\eta$ is {\em natural} from the physical point of view since the free Hamiltonian for the rescaled kinetic equation \eqref{eq:kinetic-rescaled} is given by
\[
	H(x,v) = \Phi(x) + \frac{1}{2}|v|^2\,,
\]
and $\eta$ is equal to $e^{-H}\leb^{2d}$ up to some constant.

\begin{lemma}\label{lem:A1}
	Let $T>0$ and $f$ be a solution to the following equation on the time interval $[0,T]$ with sufficient regularity and integrability:
	\begin{align*}
		\partial_t f_t + \gamma\, v\cdot\nabla_x f_t - \gamma \nabla\Phi\cdot \nabla_v f_t = \gamma^2 \nabla_v\cdot\bigl(\nabla_v f_t + v\,f_t\bigr) + \gamma \nabla_v\cdot g_t\,,
	\end{align*}
	for some $g\in L^p((0,T);L_H^p(\R^d \times \R^d))$. Further, set $w := df/d\eta = fe^H$ and $\zeta=dg/d\eta = ge^H$. 
	
	Then for any $p\in[2,\infty)$,
	\[
		\|w_t\|_{L^p(\eta)}^p  \le  \|w_0\|_{L^p(\eta)}^p  + \frac{p(p-1)}{2}\int_0^t\intrr |w_r|^{p-2}|\zeta_r|^2\, d\eta\, dr\qquad\text{for all $t \in [0,T]$\,.}
	\]
\end{lemma}
\begin{proof}
	From the equation for $f$, we find
	$$\begin{aligned}
&\frac{d}{dt}\intrr |w_t|^p\,d\eta \\
&\qquad= -\gamma p\intrr |w_t|^{p-2} w_t \lt( v \cdot \nabla_x \mu_t\rt)d\leb^{2d} +\gamma p\intrr |w_t|^{p-2} w_t \lt(\nabla \Phi \cdot \nabla_v \mu_t \rt) d\leb^{2d}\cr
&\qquad\qquad +\gamma^2 p \intrr |w|^{p-2} w_t\, \nabla_v\cdot\bigl( \nabla_v \mu_t + v\mu_t\bigr)\, d\leb^{2d} + \gamma p \intrr |w_t|^{p-2} w \lt(\nabla_v\cdot g_t\rt) d\leb^{2d} \cr
&\qquad=: \mathrm{(I)} + \mathrm{(II)} + \mathrm{(III)} + \mathrm{(IV)}\,.
\end{aligned}$$
Note that
\bq\label{rel_q}
\nabla_x f\,d\leb^{2d} = \nabla_x w \,d\eta - (\nabla \Phi) \,w \,d\eta \quad \mbox{and} \quad \nabla_v f\,d\leb^{2d} = \nabla_v w \,d\eta - v\, w \,d\eta\,.
\eq
We use this relation to rewrite $\mathrm{(I)}$ as
\begin{align*}
	\mathrm{(I)} &= -\gamma p \intrr |w_t|^{p-2}w_t \lt(v \cdot \nabla_x w\rt)d\eta + \gamma p\intrr |w_t|^p \lt(v \cdot \nabla\Phi\rt)d\eta \\
	&= \gamma (p-1)\intrr |w_t|^p \lt(v \cdot \nabla \Phi\rt)d\eta\,,
\end{align*}
where, in the last equality, we used the fact that
\[
	\intrr |w_t|^{p-2}w \lt(v \cdot \nabla_x w_t\rt) d\eta = \frac{1}{p}\intrr (\nabla_x|w_t|^{p})\cdot v\,d\eta = \frac{1}{p} \intrr |w_t|^p\lt( v\cdot \nabla\Phi\rt) d\eta\,.
\]
In a similar fashion, we use \eqref{rel_q} to rewrite $\mathrm{(II)}$ as
\[
	\mathrm{(II)} = -\gamma(p-1)\intrr |w_t|^p\lt(v\cdot\nabla\Phi\rt) d\eta\,,
\]
which, together with $\mathrm{(I)}$ implies $\mathrm{(I)} + \mathrm{(II)} = 0$. As for $\mathrm{(III)}$, we easily find
\begin{align*}
	\mathrm{(III)} &= -\gamma^2p(p-1)\intrr |w_t|^{p-2} \nabla_v w_t \cdot \lt( \nabla_v \mu_t + v \mu_t \rt) d\leb^{2d} \\
	&= -\gamma^2p(p-1)\intrr |w_t|^{p-2} |\nabla_v w_t|^2 \,d\eta\,.
\end{align*}
Finally, to estimate $\mathrm{(IV)}$, we apply the Young and H\"older inequalities to obtain
\begin{align*}
	\mathrm{(IV)} &= \gamma p(p-1) \intrr |w_t|^{p-2} \nabla_v w_t \cdot \zeta_t\, d\eta \\
	&\le \gamma p(p-1) \intrr |w_t|^{p-2} \lt(\frac{\gamma}{2}|\nabla_v w_t|^2 + \frac{1}{2\gamma}|\zeta_t|^2\rt) d\eta \\
	& \le \frac{\gamma^2 p(p-1)}{2}\intrr |w_t|^{p-2}|\nabla_v w_t|^2\,d\eta + \frac{p(p-1)}{2}\intrr |w_t|^{p-2} |\zeta_t|^2\,d\eta\,. 
\end{align*}
Putting the terms $\mathrm{(I)}$--$\mathrm{(IV)}$ together and using the identity
\[
	|\nabla_v |w_t|^{p/2}|^2 = (p^2/4) |w_t|^{2(p/2-1)}|\nabla w_t|^2 = (p^2/4)|w_t|^{p-2}|\nabla_v w_t|^2\,,
\]
we then obtain the required estimate.
\end{proof}

In the following, we assume that $\mu_t^\gamma\ll \eta$ for all $\gamma>0$ and $t\ge 0$.

\begin{lemma}\label{lem_u} Let $\mu^\gamma$ be a solution of \eqref{eq:kinetic-rescaled} with sufficient regularity. We  assume that
\[
	c_0:= \sup_{\gamma>0}\sup_{t\in[0,T]}\|\nabla K \star \rho_t^\gamma \|_{L^\infty}<\infty\,.
\]
Then for any $p\in[2,\infty)$, we have for $u_t^\gamma := d\mu_t^\gamma/d\eta = \mu_t^\gamma e^H$ (the density of $\mu_t^\gamma$ w.r.t.\ $\eta$),
\[
\|u_t^\gamma\|_{L^p(\eta)}^p  \le \|u_0^\gamma\|_{L^p(\eta)}^p + \frac{p(p-1)c_0^2}{2}\int_0^t\|u_r^\gamma\|_{L^p(\eta)}^p \,dr\,,
\]
for all $t\in[0,T]$.
%
In particular, this yields
\[
	\|\mu_t^\gamma\|_{L_H^p} \leq \|\mu_0^\gamma\|_{L_H^p}\,e^{c_0^2(p-1)t/2}\qquad\text{for all $t\in[0,T]$}\,.
\]
\end{lemma}
\begin{proof} The proof makes use of Lemma~\ref{lem:A1} with
\[
	g_t(x,v) := (\nabla K\star \rho_t^\gamma)(x)\, \mu_t^\gamma(x,v)\,.
\]
Observing that
\[
\zeta^\gamma_t(x,v) 
 = g^\gamma_t(x,v)e^{H(x,v)} = (\nabla K\star \rho_t^\gamma)(x)\, \mu_t^\gamma(x,v)e^{H(x,v)} =  (\nabla K\star \rho_t^\gamma)(x)\, u^\gamma_t(x,v)\,,
\]
and
\[
	\intrr |u_t^\gamma|^{p-2}|\zeta_t|^2\,d\eta \le c_0^2 \intrr |u_t^\gamma|^p\,d\eta\qquad\text{for almost all $t\in(0,T)$}\,,
\]
we deduce from Lemma~\ref{lem:A1},
\[
	\frac{d}{dt} \|u_t^\gamma\|_{L^p(\eta)}^p + \frac{2\gamma^2 (p-1)}{p}\lt\|\nabla_v|u_t^\gamma|^{p/2}\rt\|_{L^2(\eta)}^2 \le  \frac{p(p-1)c_0^2}{2}\|u_t^\gamma\|_{L^p(\eta)}^p\,,
\]
for almost every $t\in(0,T)$. We then conclude the proof by applying Gr\"onwall's inequality.
\end{proof}

Following similar computations as in Lemma~\ref{lem_u}, we obtain the next statement.

\begin{lemma}\label{Lp-bounds-singular}
Let $\mu^\gamma$ be a solution of \eqref{eq:kinetic-rescaled} with sufficient regularity. Further, let $p\in[2,\infty)$ such that
\[
	\|\nabla K\star \varrho \|_{L^\infty} \le c_{K,p}\lt(1 + \|\nu\|_{L_H^p}\rt),\qquad \varrho = \pi^x_\#\nu\,,\qquad \text{$\nu\in L_H^p(\R^d\times\R^d)\,.$}
\]
Then there exists $t_{p}^*\in(0,\infty)$ such that
\[
	M := \sup_{\gamma>0}\sup\nolimits_{t\in[0,t_p^*]} \|\mu_t^\gamma\|_{L_H^{p}} <\infty\,.
\]
In particular, for any $q\in[2,\infty)$, we have
\[
	\|\mu_t^\gamma\|_{L_H^q} \leq \|\mu_0^\gamma\|_{L_H^q}\,e^{c_{K,p}^2(1+M^2)(q-1)t}\qquad\text{for all $t\in[0,t_p^*]$}\,.
\]
Moreover, we have
\[
	\|\rho_t^\gamma\|_{L^q}  \leq C\qquad\text{for all $t\in[0,t_p^*]$}\,,
\]
where $C>0$ is independent of $\gamma>0$.
\end{lemma}
\begin{proof}
	As in the proof of Lemma~\ref{lem_u}, we apply Lemma~\ref{lem:A1} with
	\[
		g_t(x,v) := (\nabla K\star \rho_t^\gamma)(x)\, \mu_t^\gamma(x,v)\,.
	\]
	From the assumption on $K$, we have that
	\[
		|g_t|(x,v) \le c_{K,p}\bigl(1 + \|\mu_t^\gamma\|_{L_H^{p}}\bigr)|\mu_t^\gamma|(x,v)\,.
	\]
	Therefore, an application of Lemma~\ref{lem:A1} provides the differential inequality
	\[
		\frac{d}{dt}\|\mu_t^\gamma\|_{L_H^{p}} \le (p-1)c_{K,p}^2\Bigl(1 + \|\mu_t^\gamma\|^2_{L_H^{p}}\Bigr)\|\mu_t^\gamma\|_{L_H^{p}}\,,
	\]
	which consequently leads to the estimate
	\[
		\|\mu^\gamma_t\|_{L_H^{p}} \le e^{(p-1)c_{K,p}^2 t}\left(\frac{1}{\|\mu_0^\gamma\|_{L_H^{p}}^2} - \frac{e^{2(p-1)c_{K,p}^2 t}-1}{(p-1)c_{K,p}^2}\right)^{-1}.
	\]
	Hence, for any
	\[
		t_p^* < \frac{1}{2(p-1)c_{K,p}^2}\log\left(1 + \frac{(p-1)c_{K,p}^2}{\|\mu_0^\gamma\|_{L_H^{p}}^2}\right),
	\]
	we then obtain the required estimate.
	
	A direct application of Lemmas \ref{lem_u} and \ref{lem:lp-rho-u} yields the other estimates for $q\in[2,\infty)$, $q\ne p$.
\end{proof}

\begin{remark}\label{rem:singular-K}
Let $q\in(1,\infty)$ and $p= \max\{2,q/(q-1)\}$. If $R > 0$ exists such that
\[
	c_{K} := \|\nabla K\|_{L^q(B_R)} + \|\nabla K\|_{L^\infty(\R^d\setminus B_R)}<\infty\,,
\]
then for all $x\in\R^d$, we obtain
\begin{align*}
	(\nabla K\star\varrho)(x) &= \intr (\nabla K)(x-y)\,\varrho(dy) \le \int_{B_R(x)} |\nabla K|(x-y)\,\varrho(dy) + \|\nabla K\|_{L^\infty(\R^d\setminus B_R)}\,.
\end{align*}
The first term on the right-hand side may be further estimated by
\[
	\int_{B_R(x)} |\nabla K|(x-y)\,\varrho(dy) \le \begin{cases}
		\|\nabla K\|_{L^q(B_R)}\|\varrho\|_{L^{q/(q-1)}} &\text{for $q\in(1,2]$} \\
		c_{q,2,R}\|\nabla K\|_{L^q(B_R)}\|\varrho\|_{L^2} &\text{for $q\in(2,\infty)$}
	\end{cases}\,,
\]
where $c_{q,2,R}>0$ is a constant obtained from the standard $(L^q,L^2)$-interpolation inequality on the bounded domain $B_R$. Altogether, this gives
\[
	\|\nabla K\star\rho\|_{L^\infty} \le c_{K,q}\lt(1 + \|\nu\|_{L_H^{p}}\rt)
\]
for some constant $c_{K,q}>0$, i.e.\ the assumption on $K$ in Lemma~\ref{Lp-bounds-singular} is satisfied. Therefore, this allows us to consider the case where $\nabla K$ is singular at the origin.
\end{remark}

Next, we provide an estimate for the spatial derivative of $\mu_t^\gamma$.
 
\begin{lemma}\label{lem_xi} Let $\mu^\gamma$ be a solution of \eqref{eq:kinetic-rescaled} with sufficient regularity. We assume that 
\[
	c_0:= \sup_{t\in[0,T]}\|\nabla K \star \rho_t^\gamma \|_{L^\infty}^2\,,\quad c_1 := \sup_{t\in[0,T]}\|\sfF(\cdot,\rho_t^\gamma)\|^2_{\Lip} < \infty\qquad\text{uniformly in $\gamma>0$\,,}
\]
and denote $\xi_t^\gamma:= (\nabla_x \mu_t^\gamma)\,e^{H}$ for all $t\in[0,T]$. Then for any $p\in[2,\infty)$, we have 
$$\begin{aligned}
\intrr |\xi_t^\gamma|^p\,d\eta &\leq  \intrr |\xi_0^\gamma|^p\,d\eta + (p-1)\lt( c_1(p-2) + c_0 p\rt) \int_0^t\intrr |\xi_s^\gamma|^p \,d\eta\, ds \cr
&\quad + 2c_1(p-1)\int_0^t \|\mu_s^\gamma\|_{L_H^p}^p\,ds
\end{aligned}$$
for all $t \in [0,T]$. In particular, this implies
\[
	\sup\nolimits_{t\in[0,T]}\|\nabla_x \mu_t^\gamma\|_{L_H^p} \leq C\bigl(\|\nabla_x \mu_0^\gamma\|_{L_H^p} + \|\mu_0^\gamma\|_{L_H^p}\bigr)\,,
\]
where $C>0$ depends on $c_0, c_1, p$, and $T$, but independent of $\gamma>0$.
\end{lemma}
\begin{proof} 
For notational simplicity, we drop superscript $\gamma$ throughout this proof.

\smallskip

Set $q^k := \pa_{x_k}\mu$ for $k=1,\dots,d$ and consider $\xi^k:=q^k e^H$. Recall that $\eta = e^{-H}\leb^{2d}$. We notice from \eqref{eq:kinetic-rescaled} that $q^k$ satisfies
\[
\partial_t q_t^k + \gamma v\cdot\nabla_x q_t^k - \gamma\nabla\Phi\cdot \nabla_v q_t^k = \gamma^2\nabla_v\cdot\lt( \nabla_v q_t^k + v q_t^k\rt) + \gamma\nabla_v\cdot g_t^k\,,
\]
with
\[
	g_t^k(x,v) = (\nabla K\star\rho_t)(x)\,q_t^k(x,v) - \partial_{x_k}\sfF(x,\rho_t)\,\mu_t(x,v).
\]
By the assumptions placed on $\nabla K$ and $\partial_{x_k}\sf F$, we deduce that
\[
	|g_t^k|(x,v) \le c_0|q_t^k|(x,v) + c_1|\mu_t|(x,v)\,,
\]
i.e.\ $g^k\in L^p((0,T); L_H^p(\R^d \times \R^d))$ for $k=1,\ldots,d$. Therefore, an application of Lemma~\ref{lem:A1} gives
\[
	\frac{d}{dt} \|\xi_t^k\|_{L^p(\eta)}^p + \frac{2\gamma^2 (p-1)}{p}\lt\|\nabla_v|\xi_t^k|^{p/2}\rt\|_{L^2(\eta)}^2 \le  c_2(p)\|\xi_t^k\|_{L^p(\eta)}^p + 2c_1^2(p-1)\|\mu_t\|_{L_H^p}^p\,,
\]
with $c_2(p) := (p-1)\lt(c_0^2 p + c_1^2(p-2)\rt) > 0$.
By Lemma~\ref{lem_u}, the second term on the right-hand side of the previous inequality can be bounded from above by
\[
	2c_1^2(p-1)\|\mu_0\|_{L_H^p}^p\,e^{c_0^2p(p-1)t/2}\,.
\]
This, together with applying Gr\"onwall's lemma yields
\[
	\|q_t^k\|_{L_H^p}^p \leq \|q_0^k\|_{L_H^p}^p e^{c_2(p)t} + 2c_1^2(p-1)\|\mu_0\|_{L_H^p}^p \int_0^t e^{c_0^2p(p-1)s/2 + c_2(p)(t-s)}\,ds\,,
\]
thereby concluding the proof.
\end{proof}

In a similar fashion as in the proof of Lemma~\ref{Lp-bounds-singular}, we obtain the following result.

\begin{lemma}\label{lem_uni_sing} 
Let $\mu^\gamma$ be a solution of \eqref{eq:kinetic-rescaled} with sufficient regularity. Further, let $p\in[2,\infty)$ such that
\[
	\|\nabla K\star \varrho \|_{W^{1,\infty}} \le c_{K,p}\lt(1 + \|\nu\|_{W_{x,H}^{1,p}}\rt),\qquad \varrho = \pi^x_\#\nu\,,\qquad \text{$\nu\in W_{x,H}^{1,p}(\R^d\times\R^d)\,.$}
\]
There there exists $t_q^{**} > 0$ such that 
\[
\sup_{0 \leq t \leq t_q^{**}} \|\nabla_x \mu_t^\gamma\|_{L_H^{q'}} \leq C\,, 
\]
where $C>0$ is independent of $\gamma>0$. 
\end{lemma}

The following lemma provides a class of examples of interaction potentials $K$ that satisfies the requirements of Lemma~\ref{lem_uni_sing}.

\begin{lemma}\label{lem_klip} Let $q\in(1,\infty)$, $p=\max\{2,q/(q-1)\}$. Suppose $R\ge 1$ exists such that
\[
	c'_{K} := \|\nabla K\|_{L^q(B_{2R})} + \|\nabla^2 K\|_{L^\infty(\R^d\setminus B_R)}<\infty
	\,.
\]
Then there exists a constant $c_{K,p}'>0$ such that the following estimate holds:
\[
	\|\nabla K\star\varrho\|_{\Lip} \le c_{K,q}'\lt( 1 + \|\nu\|_{W_{x,H}^{1,p}} \rt)\,,\qquad \varrho = \pi^x_\#\nu\,,\qquad \nu\in W_{x,H}^{1,q'}(\R^d\times\R^d)\,.
\]
\end{lemma}
\begin{proof}
By density argument, it is sufficient to show the estimate for $K\in \calC_b^2(\R^d)$.
	
Let $\chi_R$ be a smooth cut-off function satisfying $1_{B_R}\le \chi_R\le 1_{B_{2R}}$ and $|\nabla\chi_R|(x)\le 1/2$ for all $x\in\R^d$. Then, for each $k,j=1,\ldots,d$, we have
\begin{align*}
&\lt|\partial_{x_k} \bigl(\partial_j K\star \rho\bigr)(x)\rt| \cr
&\quad = \lt|\intrr \partial_{x_k}(\partial_j K)(x-z)\,\mu(dzdv)\rt| \\
&\quad \le \lt|\intrr \chi_R^2(x-z)\,\partial_{z_k}(\partial_j K)(x-z)\,\mu(dzdv)\rt| + \lt|\intr (1-\chi_R^2)(x-z)(\partial_{x_k}\pa_j K)(x-z)\,\rho(dz)\rt| \\
&\quad =: \mathrm{(I)} + \mathrm{(II)},
\end{align*}
where we use the integration by parts to estimate
\begin{align*}
\mathrm{(I)} &\leq 2\intrr |\chi_R(x-z)||\nabla \chi_R(x-z)| |(\partial_j K)(x-z)|\, \mu(dzdv) \cr
&\quad + \intrr \chi_R^2(x-z)|(\partial_j K)(x-z)||(\partial_{z_k}\mu)(z,v)|\, \leb^{2d}(dzdv)\cr
&\leq \begin{cases}
 	\|\partial_j K\|_{L^q(B_{2R})}\bigl( \|\mu\|_{L^{q'}} + \|\partial_{x_k}\mu\|_{L^{q'}}\bigr)	&\text{for $q\in(1,2]$} \\
 	c_{q,2,R}\|\partial_j K\|_{L^q(B_{2R})}\bigl( \|\mu\|_{L^2} + \|\partial_{x_k}\mu\|_{L^2}\bigr)	&\text{for $q\in(2,\infty)$}
 \end{cases}\,.
\end{align*}
Here $q'=q/(q-1)$, and $c_{q,2,R}>0$ is a constant obtained from the standard $(L^q,L^2)$-interpolation inequality on the bounded domain $B_{2R}$. Since $\|\mu\|_{L^{p}} \leq \|\mu\|_{L_H^{p}}$ and $\|\partial_{x_k}\mu\|_{L^{p}} \leq  \|\partial_{x_k}\mu\|_{L_H^{p}}$, we obtain
\[
\mathrm{(I)} \leq \max\{1,c_{q,2,R}\}\|\partial_j K\|_{L^q(B_{2R})} \lt( \|\partial_{x_k} \mu\|_{L_H^{p}} + \|\mu\|_{L_H^{p}}\rt)\,.
\]
As for $\mathrm{(II)}$, we easily obtain
\[
\mathrm{(II)} \leq \|\partial_k\partial_j K\|_{L^\infty(\R^d\setminus B_R)} \intr (1-\chi_R^2)(x-z)\,\rho(dz) \leq \|\partial_k\partial_j K\|_{L^\infty(\R^d\setminus B_R)}\,.
\]
Putting the estimates together and summing up in $k,j=1,\ldots,d$ yields
\[
	\|\nabla^2 K\star \rho\bigr\|_{L^\infty} \le c_{K,q}'\lt( 1 + \|\nabla_x \mu\|_{L_H^{p}} + \|\mu\|_{L_H^{p}}\rt),
\]
with the constant $c_{K,q}'=\max\{1,c_{q,2,R}\}c_K'$.
\end{proof}

\begin{proof}[Proof of Proposition~\ref{lem:regularity-summary}]
	The proof follows from a combination of Lemmas~\ref{lem_u}, \ref{Lp-bounds-singular}, \ref{lem_xi}, and \ref{lem_uni_sing}.
\end{proof}

%
%
%
%
%
%
%
\section{Stability for singular interaction forces}\label{sec:SIP}
 

In this section, we deal with the singular interaction potentials and provide the details of the proofs for Theorem~\ref{thm:stability-limit}. We begin by providing a Lipschitz-type estimate of the interaction forces which will be crucially used for the stability estimate.

\begin{lemma}\label{lem:K-infty-lipschitz}
Let $q \in (1,\infty]$. Suppose that there exists $R \in (0,\infty)$ such that
\[
    c_K:=\|\nabla K\|_{L^q(B_{2R})} + \|\nabla^2K\|_{L^\infty(\R^d\setminus B_R)} <\infty\,.
\]
Then for any $\mu^\gamma\in W_{x,H}^{1,q'}(\R^d\times\R^d)$, $q'=q/(q-1)$, and
\[
	M(\mu^\gamma):= \|\mu^\gamma\|_{W_{x,H}^{1,q'}} + \intrr |v|\,\mu^\gamma(dxdv) <\infty\,,
\]
the following estimate holds:
\[
|(\nabla K\star \rho^\gamma)(x-v/\gamma) - (\nabla K\star \bar\rho^\gamma)(x)| \le  \frac{c_0}{\gamma}(1+|v|)\lt(1+M(\mu^\gamma)\rt)
\]
for some constant $c_0>0$, independent of $\mu^\gamma$ and $\gamma>0$. 
\end{lemma}
\begin{proof} By density argument, it is sufficient to show the estimate for $K\in \calC_b^2(\R^d)$. 

We begin by writing
\begin{align*}
&(\nabla K\star\rho^\gamma)(x-v/\gamma) - (\nabla K\star\bar\rho^\gamma)(x) \\
&\quad =\intrr \Bigl((\nabla K)(x-v/\gamma-z) - (\nabla K)(x- w/\gamma - z)\Bigr)\,\mu^\gamma(dzdw) \\
&\quad =\frac{1}{\gamma}\int_0^1 \intrr  (D^2 K)((x-z) + \theta(w-v)/\gamma)) (w-v)\,\mu^\gamma(dzdw)\,d\theta\,.
\end{align*}
In the following, we set $A_r=A_r^{\theta,\gamma}(x,v,w):=B_r(x+\theta(w-v)/\gamma)$ for each $x,v,w\in\R^d$ and $r\ge 1$. Now consider a smooth cut-off function $\chi_R$ satisfying $1_{B_R}\le \chi_R\le 1_{B_{2R}}$ for all $x\in\R^d$. Denoting $u^\gamma := \mu^\gamma e^{H}$ and using the shorthand $\bar x=x+\theta(w-v)/\gamma$, we obtain for every $k,j=1,\ldots,d$ and almost every $w\in\R^d$,
\begin{align*}
\begin{aligned}
&\intr (\partial_{x_k}\partial_{j}K)(\bar x-z)\,u^\gamma(z,w)\,\sigma_\Phi(dz)\\
&\qquad = -\intr \chi_R^2(\bar x-z) (\partial_{z_k} \partial_jK)(\bar x-z)\,u^\gamma(z,w)\,\sigma_\Phi(dz) \\
&\quad\qquad +\intr (1-\chi_R^2(\bar x-z))(\partial_{z_k} \partial_jK)(\bar x-z)\,u^\gamma(z,w)\,\sigma_\Phi(dz) \\
&\qquad = \intr \chi_R^2(\bar x-z) (\partial_jK)(\bar x-z)(\partial_{z_k}u^\gamma)(z,w)\,\sigma_\Phi(dz) \\
&\quad\qquad - \intr \chi_R^2(\bar x-z) (\partial_j K)(\bar x-z)(\partial_k\Phi)(z)\,u^\gamma(z,w)\,\sigma_\Phi(dz) \\
&\quad\qquad -2\intr \chi_R(\bar x-z)(\partial_k\chi_R)(\bar x-z)(\partial_j K)(\bar x-z)\,u^\gamma(z,w)\,\sigma_\Phi(dz) \\
&\quad\qquad +\intr (1-\chi_R^2(\bar x-z))(\partial_{z_k} \partial_jK)(\bar x-z)\,u^\gamma(z,w)\,\sigma_\Phi(dz)\,.
\end{aligned}
\end{align*}
Here, the last term on the right-hand side of the above can be bounded by
\[
	\|\nabla^2K\|_{L^\infty(\R^d\setminus B_R)}\|u^\gamma(\cdot,w)\|_{L^1(\sigma_\Phi)} \leq 2c_K\|u^\gamma(\cdot,w)\|_{L^1(\sigma_\Phi)},
\]
since $ 1-\chi_R^2 \leq 1$. As for the first and third term, we can simply use H\"older's inequality to bound them from above by
\[
(1 + 2\|\nabla \chi_R\|_{L^\infty})\|\nabla K\|_{L^q(B_{2R})} \|u^\gamma(\cdot,w)\|_{W_x^{1,q'}(\sigma_\Phi)} \leq C_{K,R}\|u^\gamma(\cdot,w)\|_{W_x^{1,q'}(\sigma_\Phi)}.
\]
We then estimate the second term as
$$\begin{aligned}
&\lt|\intr \chi_R^2(\bar x-z) (\partial_j K)(\bar x-z)(\partial_k\Phi)(z)\,u^\gamma(z,w)\,\sigma_\Phi(dz) \rt|\cr
&\quad \leq \lt(\intr |\chi_R^2(\bar x-z) (\partial_j K)(\bar x-z)(\partial_k\Phi)(z)|^q e^{-\Phi(z)}\,\leb^{d}(dz)  \rt)^{1/q}\|u^\gamma(\cdot,w)\|_{L_x^{q'}(\sigma_\Phi)}\cr
&\quad \leq c_{\Phi,q}^{1/q} \|\nabla K\|_{L^q(B_{2R})}\|u^\gamma(\cdot,w)\|_{L_x^{q'}(\sigma_\Phi)},
\end{aligned}$$
where $c_{\Phi,q} > 0$ is the same constant that appears in ($\textbf{A}_\Phi^2$), and we used
$$\begin{aligned}
\intr |\chi_R^2(\bar x-z) (\partial_j K)(\bar x-z)(\partial_k\Phi)(z)|^q\, \sigma_\Phi(dz) &= \int_{B_R(\bar x)} |(\partial_j K)(\bar x-z)|^q |(\partial_k\Phi)(z)|^q\, \sigma_\Phi(dz) \cr
&\leq c_{\Phi,q}\int_{B_R(\bar x)} |(\partial_j K)(\bar x-z)|^q\,\leb^{d}(dz) \\
&\leq c_{\Phi,q} \|\nabla K\|_{L^q(B_{2R})}^q.
\end{aligned}$$
Hence, for almost every $w\in \R^d$, we find that
\begin{align*}
\lt|\intr (D^2K)(\bar x - z)\,u^\gamma(z,w)\,\sigma_\Phi(dz)\rt| &\le c_{R,K,\Phi,q}\Bigl(\|u^\gamma(\cdot,w)\|_{W_x^{1,q'}(\sigma_\Phi)} +  \|u^\gamma(\cdot,w)\|_{L^1(\sigma_\Phi)}\Bigr) \,.
\end{align*}
Furthermore, since for any $f\in L^r(\eta)$, $r>1$,
\begin{align*}
&\intr |w-v| \|f(\cdot,w)\|_{L^r(\sigma_\Phi)}\,\scrN^d(dw) \\
&\qquad \le \intr |w|\|f(\cdot,w)\|_{L^r(\sigma_\Phi)}\,\scrN^d(dw) + |v|\intr\|f(\cdot,w)\|_{L^r(\sigma_\Phi)}\,\scrN^d(dw) \\
&\qquad \le \left(\left(\intr|w|^{r'}\scrN^d(dw)\right)^{1/r'} + |v|\right)\|f\|_{L^r(\eta)} \le c_{\scrN,r}(1+|v|)\|f\|_{L^r(\eta)}\,,
\end{align*}
where $c_{\scrN,r}>0$ is a constant depending only on $\scrN^d$ and $r$, and for $r=1$,
\[
\intr |w-v| \|u^\gamma(\cdot,w)\|_{L^1(\sigma_\Phi)}\,\scrN^d(dw) \le |v| + \intrr |w|\,\mu^\gamma(dzdw)\,,
\]
we then obtain
\begin{align*}
&\intrr (D^2 K)(\bar x- z)\cdot(w-v)\,\mu^\gamma(dz dw)\cr
&\hspace{6em} = \intr (w-v)\cdot\left(\intr (D^2K)(\bar x- z)\,u^\gamma(z,w)\,\sigma_\Phi(dz) \right)\scrN^d(dw) \\
&\hspace{6em} \le c_0(1+|v|)\lt(1 + \|u^\gamma\|_{W_x^{1,q'}(\eta)} + \intrr |w|\,\mu^\gamma(dzdw)\rt),
\end{align*}
for the constant $c_0 >0$ independent of $u^\gamma$ and $(x,v)\in\R^d\times\R^d$. This completes the proof.
\end{proof}

We next estimate $\|\mathsf{e}_t^\gamma\|_{L^2(\bar\rho_t^\gamma)}$ in the stability estimate \eqref{eq:W2-stability}.

\begin{lemma}\label{lem_e} Under the same assumptions in Lemma \ref{lem:K-infty-lipschitz}, we have
\[
\|\mathsf{e}_t^\gamma\|_{L^2(\bar\rho^\gamma_t)}^2  \leq \frac{C}{\gamma^2}\lt(1+M(\mu_t^\gamma)\rt)^2,
\]
where $C>0$ is independent of $\gamma$.
\end{lemma}
\begin{proof}
This follows from a direct application of Lemmas~\ref{lem_ee} and \ref{lem:K-infty-lipschitz}.
\end{proof}

We can finally provide the details to the proof of Theorem~\ref{thm:stability-limit}.

\begin{proof}[Proof of Theorem~\ref{thm:stability-limit}]
	In the case {\em (i)} (respectively {\em (iii)}), Theorem~\ref{thm:DK-interaction}{\em (2)} (respectively {\em (3)}) provides an estimate for the modulated interaction energy: 
	\[
		-\scrD_K(\bar\mu_t^\gamma,\mu_t)\le c_1 \d_2^2(\bar\mu_t^\gamma,\mu_t)\,,
	\]
	for a constant $c_1>0$, independent of $\gamma\ge 1$. Clearly, the inequality above also holds true for the case {\em (ii)}, since in this case, $\scrD\ge 0$. On the other hand, Lemma~\ref{lem_e} gives us a handle on the error term $\|\mathsf{e}_t^\gamma\|_{L^2(\bar\rho_t^\gamma)}$.
	
	Together with the well-posedness result provided by Theorem~\ref{thm:well-posedness}, the general stability inequality \eqref{eq:W2-stability} then gives
\[
	\frac{1}{2}\frac{d}{dt}\d_2^2(\bar\rho_t^\gamma,\rho_t) \le  c_1\lt( \d_2^2(\bar\rho_t^\gamma,\rho_t)  + \frac{1}{\gamma^2}\rt)\qquad\text{for all $t\in[0,T_p]$},
\]
for some constant $c_1>0$, independent of $\gamma\ge 1$ and $t\in[0,T_p]$. Applying Gr\"onwall's lemma yields
\[
\sup_{0 \leq t \leq T_p}\d_2^2(\bar\rho_t^\gamma,\rho_t) \leq c_2\lt(\d_2^2(\bar\rho_0^\gamma,\rho_0) + \frac{1}{\gamma^2} \rt),
\]
for some constant $c_2>0$ is independent of $\gamma\ge 1$. Finally, using a similar argument as in the proof of Theorem~\ref{thm_reg}, we conclude that
\[
\sup_{0 \leq t \leq T_p}\d_2^2(\rho^\gamma_t,\rho_t) \leq C\lt(\d_2^2(\rho_0^\gamma,\rho_0) + \frac{1}{\gamma^2}\rt),
\]
where $C>0$ is independent of $\gamma\ge 1$.
\end{proof}

\section{Bounded interaction forces}\label{sec:bdd}


As mentioned in the introduction, we will employ a different approach to obtain stability estimates when $\nabla K$ is only assumed to be bounded. In any case, we will need a Lipschitz-type estimate in order to control the error term $\mathsf{e}_t^\gamma$. Here, we will make an additional assumption on the external potential $\Phi$.

\begin{lemma}\label{lem_wreg0} Let $\nabla K\in L^\infty(\R^d)$ and $\Phi$ satisfy additionally $\sigma_\Phi(\R^d)<\infty$.

Then for any $\mu^\gamma\in\calP(\R^d\times\R^d)$ with $\nabla_x\mu^\gamma\in L_H^2(\R^d\times\R^d)$,
the following estimate holds:
\[
|(\nabla K\star \rho^\gamma)(x-v/\gamma) - (\nabla K\star \bar\rho^\gamma)(x)| \le  \frac{c_0}{\gamma}(1+|v|)\|\nabla_x \mu^\gamma\|_{L_H^2}
\]
for some constant $c_0>0$, independent of $\mu^\gamma$ and $\gamma>0$.
\end{lemma}
\begin{proof} 
The proof is very similar to that of Lemma \ref{lem:K-infty-lipschitz}. Again, by a density argument, it is sufficient to show the estimate for $K\in \calC_b^2(\R^d)$. Elementary computations give us the following estimate:
\begin{align*}
	&\lt|\int_0^1\intrr (\nabla^2 K)((x-z) + \theta(w-v)/\gamma))\,(w-v)\,\mu^\gamma(dzdw) \rt| \\
	&\qquad = \lt| \int_0^1 \intrr (\nabla K)((x-z) + \theta(w-v)/\gamma))(w_k-v_k) \,(\nabla_z\mu^\gamma)(z,w)\,\leb^{2d}(dzdw)\rt| \\
	&\qquad \le \|\nabla K\|_{L^\infty}\intrr |w-v||(\nabla_z\mu^\gamma)(z,w)|\,\leb^{2d}(dzdw) \\
	&\qquad \le (2\pi)^{d/2}\sigma_\Phi(\R^d)\|\nabla K\|_{L^\infty}\|\nabla_z\mu^\gamma\|_{L_H^2}  \intr |w-v|\, \scrN^d(dw) \\
	&\qquad \le c_0(1+|v|)\|\nabla_z\mu^\gamma\|_{L_H^2}\,,
\end{align*}
for some constant $c_0>0$ depending only on $d$, $\sigma_\Phi(\R^d)$ and $\|\nabla K\|_{L^\infty}$. The rest of the proof follows along the same line as that of Lemma~\ref{lem:K-infty-lipschitz}.
\end{proof}

\begin{remark}
	It is not difficult to see that the additional assumption on $\Phi$ in Lemma~\ref{lem_wreg0} is satisfied for potentials $\Phi$ satisfying $\Phi(x)/|x|\to \infty$ as $|x|\to\infty$.
\end{remark}

The following proposition provides a counterpart of the general stability inequality \eqref{eq:W2-stability} in the context of the relative entropy, which makes use of a result established in \cite[Theorem 2.18]{DLPS18}.

\begin{proposition}\label{prop_re}
	Let $\nabla K\in L^\infty(\R^d)$ and $\Phi$ satisfy additionally $\sigma_\Phi(\R^d)<\infty$. 
	
	Furthermore, let $(\mu_t^\gamma)_{t\in[0,T]}$ be a weak solution of \eqref{eq:kinetic-rescaled} satisfying
	\[
		\sup_{\gamma\ge 1}\sup_{t\in[0,T]}\|\nabla_x \mu_t^\gamma\|_{L_H^2} <\infty\,.
	\]
	If $(\rho_t)_{t\in[0,T]}$ is the $\scrE$-regular solution of \eqref{eq:diffusion}.
%
%
	Then,
	\[
		\Ent(\bar\rho_t^\gamma|\rho_t) \le \lt(\Ent(\bar\rho_0^\gamma|\rho_0) + \frac{\beta t}{4\gamma^2}\rt) e^{(\lambda/4) t}\qquad\text{for all $t\in[0,T]$\,,}
	\]
	for some appropriate constants $\lambda,\beta>0$, independent of $\gamma\ge 1$,
%
\end{proposition}
\begin{proof} Under the assumption on $\Phi$ and $K$, the intermediate system \eqref{eq:aux-rho} may be written as
\[
	\partial_t\bar\rho_t^\gamma + \nabla\cdot(\bar\rho_t^\gamma \sfF(\cdot,\rho_t)) - \Delta\bar\rho_t^\gamma = -\nabla\cdot(\bar\rho_t^\gamma h_t^\gamma)\,,
\]
with $h_t^\gamma:=d\bar\jmath_t^\gamma/d\bar\rho_t^\gamma - \sfF(\cdot,\rho_t)$. Note that for any $\xi\in\calC_c^\infty(\R^d)$,
	\begin{align*}
		\intr \xi\cdot \bar\jmath_t^\gamma(dx) = \intrr \langle\xi(x),\sfF(x-v/\gamma,\rho_t^\gamma)\rangle\,(\Gamma^\gamma_\#\mu_t^\gamma)(dxdv) \le \|\xi\|_{L^2(\bar\rho_t^\gamma)}\|\sfF(\cdot,\rho_t^\gamma)\|_{L^2(\rho_t^\gamma)}\,.
	\end{align*}
	Thus, by duality, we find that $\bar\jmath_t^\gamma \ll \bar\rho_t^\gamma$, and the following estimate holds:
	\[
		\|d\bar\jmath_t^\gamma/d\bar\rho_t^\gamma\|_{L^2(\bar\rho_t^\gamma)} \le \|\sfF(\cdot,\rho_t^\gamma)\|_{L^2(\rho_t^\gamma)}\qquad\text{for all $t\in[0,T]$}\,.
	\]
Therefore, we have that
\[
	\|h_t^\gamma\|_{L^2(\bar\rho_t^\gamma)} \le \|\sfF(\cdot,\rho_t^\gamma)\|_{L^2(\rho_t^\gamma)} + \|\sfF(\cdot,\rho_t)\|_{L^2(\bar\rho_t^\gamma)}\qquad\text{for all $t\in[0,T]$\,.}
\]
An application of \cite[Theorem 2.18]{DLPS18} then yields the entropy estimate
\begin{align}\label{eq:entropy-inequality}
	\Ent(\bar\rho_t^\gamma|\rho_t) \le \Ent(\bar\rho_0^\gamma|\rho_0) + \frac{1}{4}\int_0^t \|h_r^\gamma\|_{L^2(\bar\rho_r^\gamma)}^2\,dr\,.
\end{align}
We now provide a more precise estimate of $\|h_t^\gamma\|_{L^2(\bar\rho_t^\gamma)}$. For any $\xi\in\calC_c^\infty(\R^d)$, we have
\begin{align*}
	\intr \xi(x)\,h_t^\gamma(x)\,\bar\rho_t^\gamma(dx) &= -\intrr \xi(x)\Bigl[ \sfF(x,\rho_t) - \sfF(x-v/\gamma,\rho_t^\gamma)\Bigr](\Gamma^\gamma_\#\mu_t^\gamma)(dxdv) \\
	&= -\intr \xi(x)\Bigl[ \sfF(x,\rho_t) - \sfF(x,\bar\rho_t^\gamma)\Bigr]\,\bar\rho_t^\gamma(dx) \\
	&\hspace{2em} - \intrr \xi(x)\Bigl[\sfF(x,\bar\rho_t^\gamma)-\sfF(x-v/\gamma,\rho_t^\gamma)\Bigr]\, (\Gamma^\gamma_\#\mu_t^\gamma)(dxdv) \\
	&=: \mathrm{(I)} + \mathrm{(II)}\,.
\end{align*}
As for (I), we use the estimate
\[
\|\nabla K\star \bar\rho_t^\gamma - \nabla K\star\rho_t\|_{L^\infty} 	\le \|\nabla K\|_{L^\infty} \|\bar\rho_t^\gamma - \rho_t\|_{L^1} \le \sqrt{2}\|\nabla K\|_{L^\infty}\sqrt{\Ent(\bar\rho_t^\gamma|\rho_t)}\,,
\]
to obtain
\[
	|\mathrm{(I)}|\le \sqrt{2}\|\nabla K\|_{L^\infty}\|\xi\|_{L^2(\bar\rho_t^\gamma)}\sqrt{\Ent(\bar\rho_t^\gamma|\rho_t)}\,.
\]
In the former estimate, we used the well-known Csisz\`{a}r--Kullback--Pinsker inequality:
\[
	\|\mu - \nu\|_{L^1}\le \sqrt{2\Ent(\mu,\nu)}\qquad\text{for any $\mu,\nu\in L^1(\R^d)$\,.}
\]
As for the second term, we apply Lemma \ref{lem_wreg0} to obtain
$$\begin{aligned}
|\sfF(x-v/\gamma,\rho_t^\gamma)-\sfF(x,\bar\rho_t^\gamma)|^2 
&\leq 2\frac{\|\nabla \Phi\|_{\Lip}^2}{\gamma^2}|v|^2 + \frac{4 c_0^2}{\gamma^2}\lt(1+ |v|^2\rt)\|\nabla_x \mu^\gamma\|_{L_H^2}^2\cr
&\leq \frac{c^2}{\gamma^2} (1 + |v|^2)\,,
\end{aligned}$$
for some constant $c>0$ independent of $\gamma\ge 1$. Consequently, we obtain
$$\begin{aligned}
|\mathrm{(II)}| & \leq \|\xi\|_{L^2(\bar\rho_t^\gamma)}\sqrt{\iint_{\R^d \times \R^d} |F(x,\bar\rho_t^\gamma)-F(x-v/\gamma,\rho_t^\gamma)|^2\, (\Gamma^\gamma_\#\mu_t^\gamma)(dxdv)}\cr
&\leq \frac{c}{\gamma}\|\xi\|_{L^2(\bar\rho_t^\gamma)}\sqrt{1 + \intrr |v|^2\,\mu_t^\gamma(dxdv)}\,,
\end{aligned}$$
which altogether, yields,
\[
	\|h_t^\gamma\|_{L^2(\bar\rho_t^\gamma)}^2 \le \lambda\Ent(\bar\rho_t^\gamma|\rho_t) + \frac{\beta}{\gamma^2}\,,
\]
for appropriate constants $\lambda,\beta>0$ independent of $\gamma\ge 1$ and $t\in[0,T]$\,. Inserting this into the entropy inequality \eqref{eq:entropy-inequality} and applying Gr\"onwall's inequality then yields the desired result.
\end{proof}

\begin{proof}[Proof Theorem \ref{thm_wreg}] We first notice that the bounded Lipschitz distance between $\mu$ and $\nu$ can be bounded either by the $1$-Wasserstein distance between them or the $L^1$-norm of $\mu - \nu$. This, together with the monotonicity of $\d_p$-distance and the Csisz\`{a}r--Kullback--Pinsker inequality, yields
$$\begin{aligned}
\d_{\text{BL}}(\rho^\gamma_t, \rho_t) &\leq \d_{\text{BL}}(\rho^\gamma_t, \bar\rho_t^\gamma) + \d_{\text{BL}}(\bar\rho^\gamma_t, \rho_t)\cr
&\leq \d_{1}(\rho^\gamma_t, \bar\rho_t^\gamma) + \|\bar\rho_t^\gamma - \rho_t\|_{L^1}\cr
&\leq \d_2(\rho^\gamma_t, \bar\rho_t^\gamma) + \sqrt{2 \Ent(\bar\rho_t^\gamma|\rho_t)}.
\end{aligned}$$
We finally use Lemma \ref{lem_rho2} and Proposition \ref{prop_re} to conclude 
\[
	\d_{\text{BL}}^2(\rho^\gamma_t, \rho_t) \leq C\lt(\Ent(\bar\rho_0^\gamma|\rho_0) + \frac{1}{\gamma^2}\rt)\qquad\text{for all $t\in[0,T]$\,,}
\]
where $C>0$ is independent of $\gamma\ge 1$.
\end{proof}

\appendix

%
%
%
%
\section{Well-posedness for the kinetic equation} \label{app:WP}
In this appendix, we prove the well-posedness of our main equation in the solution space $W^{1,p}_{x,H}(\R^d\times\R^d)$ and simultaneously justify the uniform-in-$\gamma$ estimates obtained in Section~\ref{sec:uniform-gamma}. 

Let us begin by recalling our main equation:
\bq\label{app_kin}
\pa_t \mu_t + \gamma\,v \cdot \nabla_x \mu_t + \gamma\,\sfF(\cdot,\rho_t) \cdot \nabla_v \mu_t = \gamma^2\nabla_v \cdot (\nabla_v \mu_t + v\mu_t),
\eq
where $\sfF(\cdot,\rho_t)  = -\nabla \Phi - \nabla K \star \rho_t$, with $\rho_t = \pi^x_\#\mu_t$ being the $x$-marginal of $\mu_t$.

As mentioned in the introduction, there is available literature on the existence theory for weak solutions to the equation \eqref{app_kin} when the interaction potential $K$ is bounded and Lipschitz continuous or it has the form of $K(x) = 1/|x|^\alpha$ with $\alpha \in (0,d-2]$. The existence of classical solutions is also well studied for the equation \eqref{app_kin} without the potential $\Phi$ when $\pm\Delta K =  \delta_0$. On the other hand, we have been unable to find the proper literature that provides the well-posedness for our main equation; the case $K(x) = 1/|x|^\alpha$ with $\alpha \in (0,d-2]$ 
and $\Phi$ satisfying the assumptions presented in Section \ref{sec:assumption-results}. For that reason, we provide a well-posedness theory for \eqref{app_kin}, which is summarized in the following theorem.

\begin{theorem}\label{thm:main-app}
	Let $T>0$, $k\in\{0,1\}$ and the potentials $\Phi$ and $K$ satisfy $(\textrm{\bf A}_\Phi)$ and
	\begin{align}\label{thm:eq:main-app}
		\|\nabla K\|_{L^q(B_R)} + \|\nabla K\|_{L^\infty(\R^d\setminus B_R)} <\infty\qquad\text{for some $R>0$ and $q\in(1,\infty]$\,.}
	\end{align}
	If the initial data $\mu_0^\gamma$, $\gamma\ge 1$ satisfies
	\[
		\sup_{\gamma\ge 1} \left\{\Ent(\mu_0^\gamma|\scrN^{2d}) + \|\mu_0^\gamma\|_{W^{k,p}_{x,H}}\right\} < \infty\,,
	\]
	with $H(x,v) = \Phi(x) + \frac{1}{2}|v|^2$ and $p\ge \max\{2,q/(q-1)\}$, then there exists a positive time $T_p\in(0,T]$ and a weak solution $\mu^\gamma\in \calC([0,T_p];\calP(\R^d\times\R^d))$ of \eqref{app_kin} with
	\[
		\sup\nolimits_{t\in[0,T_p]} \Bigl( \Ent(\mu_t^\gamma|\scrN^{2d}) + \|\mu_t^\gamma\|_{W^{k,p}_{x,H}}\Bigr) < \infty\qquad\text{uniformly in $\gamma\ge 1$\,.}
	\]
	
	In the case $k=1$, the weak solution $\mu^\gamma$ is unique.
\end{theorem}

\smallskip

We begin our treatise by regularizing the potentials, $\Phi^\e = \Phi \star \delta_\e$ and $K^\e = K \star \delta_\e$ with the standard mollifier $\delta_\e \in \calC_c^\infty(\R^d)$, $\e > 0$. Since $\nabla \Phi^\e, \nabla K^\e \in \calC_b^\infty(\R^d)$ and the term $\nabla \Phi^\e \cdot \nabla_v \mu_t$ is linear, one can employ the well-posedness theory from \cite{VO90} to construct classical solutions to the following regularized equation:
\bq\label{app_rkin}
\pa_t \mu^\e_t + \gamma\,v \cdot \nabla_x \mu^\e_t + \gamma\,\sfF^\e(\cdot,\rho^\e_t) \cdot \nabla_v \mu^\e_t = \gamma^2\nabla_v \cdot (\nabla_v \mu^\e_t + v\mu^\e_t),
\eq
where $\sfF^\e(\cdot,\rho_t)  = -\nabla \Phi^\e - \nabla K^\e \star \rho_t$. More precisely, we have the following theorem (adapted from \cite[Theorem III.2]{VO90}):
\begin{theorem}\label{thm_ext_rkin} Let $T>0$ and $\e>0$. Suppose that the initial data $\mu_0^\e$ satisfies
\[
\mu_0^\e\in \mc_c^\infty(\R^d \times \R^d)\,.
\]
Then there exists a unique classical solution to the regularized equation \eqref{app_rkin} on the time interval $[0,T]$ satisfying
\begin{align}\label{thm_ext_rkin:est}
(1 + |x|^2 + |v|^2)\bigl(|\mu_t^\e| + |\nabla \mu_t^\e| \bigr) \in (L^1\cap L^\infty)(\R^d\times\R^d)\qquad\text{for all $t \in [0,T]$.}
\end{align}
\end{theorem}

\begin{remark}\label{rem:hypoelliptic}
	As a consequence of Theorem~\ref{thm_ext_rkin} and the hypoelliptic property of \eqref{app_kin} (cf.\ \cite{Hormander:1967aa}), we easily deduce that $\mu_t^\e\in \calC^\infty(\R^d\times\R^d)$ for all $t\in(0,T)$.		
\end{remark}

\subsection{Uniform in $\e$ and $\gamma$ estimates}\label{app:uniform-gamma} 

\subsubsection{Uniform moment estimates} Our first result provides moment estimates for the solution of \eqref{app_rkin} that are uniform in $\e$ and $\gamma$, and is a restatement of Lemma~\ref{lem:g-estimate}. Recall the coarse-graining map $\Gamma^\gamma(x,v)=(x+v/\gamma,v)$ and its inverse $\Gamma^{-\gamma}(\zeta,v) = (\zeta-v/\gamma,v)$ defined in Section~\ref{sec:IS}. 

\begin{theorem}\label{thm:entropy-linear}
	Let $\e>0$, $\gamma\ge 1$ and let $\mu_0^\e\in \mc_c^\infty(\R^d \times \R^d)$ satisfy
	\[
		\Ent(\mu_0^\e|\Gamma^{-\gamma}_\#\scrN^{2d}) <\infty\,,
	\]
	 and let $(\mu_t^\e)_{t\ge 0}\subset \calP(\R^d\times\R^d)$ be a solution of \eqref{eq:kinetic-rescaled} satisfying \eqref{thm_ext_rkin:est}. Further, let
	 \[
	 	\sup\nolimits_{t\in [0,T]}|\sfF^\e(x,\rho_t^\e)| \le c_\Phi(1 +|x|) + c_0
	 \]
	 for some constants $c_\Phi,c_0>0$, independent of $\e>0$ and $\gamma>0$. Then for every $t\in[0,T]$,
	\begin{align}\label{eq:entropy-linear}
		\Ent(\mu_t^\e|\Gamma^{-\gamma}_\#\scrN^{2d}) + \alpha\gamma^2\int_0^t\mathcal{D}(\mu_r^\e)\,dr \le M(\mu_0^\gamma,t)
	\end{align}    
	with $M(\mu_0^\e,t) = \Ent(\mu_0^\e|\Gamma^{-\gamma}_\#\scrN^{2d})e^{\lambda t} + (\beta/\lambda)(e^{\lambda t}-1)$ for some constants $\alpha,\beta,\lambda>0$ that are independent of $\e>0$ and $\gamma \ge 1$, and where
	\[
		\mathcal{D}(\mu) := \begin{cases}
			\iint_{\{\mu>0\}} \frac{|\nabla_v \mu + v\mu|^2}{\mu}\,d\leb^{2d} &\text{if $\nabla_v \mu + v\mu \ll \mu$\,,} \\
			+\infty &\text{otherwise.}
		\end{cases}
	\]
	In particular, we have that $\nabla_v \mu_t^\e + v\mu_t^\e \ll \mu_t^\e$ for almost every $t\in(0,T)$. 
\end{theorem}
\begin{proof}
	For notational simplicity, we will drop the superscript $\e$ within the proof.
	
	Since $(\Gamma^{-\gamma}_\#\scrN^{2d})(x,v) = \scrN^{2d}(x+v/\gamma,v)$, we set $H(x,v) = (|x+v/\gamma|^2 + |v|^2)/2$. For each $R>0$, consider the function $\eta_R(x,v) := \exp(- H\wedge R)\in W^{1,\infty}(\R^d\times \R^d)$, which satisfies $\exp(-R)\le \eta_R \le 1$ and $\nabla\log \eta_R\in L^\infty(\R^d\times\R^d)$. Further, for each $\delta>0$, we set $\psi_\delta(z) = z\log(z + \delta)$, $z\in[0,\infty)$. Note that $\psi_\delta\ge \psi$ for any $\delta>0$. Furthermore, it is immediate to see that $\psi_\delta$ is convex with
	\[
		\psi_\delta'(z) = \log(z +\delta) + \frac{z}{z + \delta}\,, 
		\qquad \psi_\delta''(s) 
		= \frac{1}{z+\delta}\lt( 1 + \frac{\delta}{z+\delta}\rt),\qquad z\ge 0,\;\delta\in(0,1)\,,
	\]
	which consequently implies, for any $0<s<t<T$,
	\[
	    \lt|\intrr \psi_\delta(\mu_t/\eta_R)\,d \eta_R - \intrr \psi_\delta(\mu_s/\eta_R)\,d \eta_R \rt| \le 2(1 + \delta^{-1}) e^R\sup_{r\in[s,t]}\|\partial_r\mu_r\|_{L^\infty} |t-s|\,,
	\]
	and hence the map $(0,T)\ni t\mapsto \intrr \psi_\delta(\mu_t/\eta_R)\,d\eta_R$ is absolutely continuous.
	
	At each point of differentiability $t\in(0,T)$, we have
	\begin{align*}
		&\intrr \psi_\delta(\mu_t/\eta_R)\,d\eta_R - \intrr \psi_\delta(\mu_0/\eta_R)\,d\eta_R \\
		&\hspace{2em}= \gamma\int_0^t\intrr v\cdot \nabla_x \psi_\delta'(\mu_r/\eta_R) \,d\mu_r\,dr - \gamma^2\int_0^t\intrr \nabla_v\psi_\delta'(\mu_r/\eta_R)\cdot[ \nabla_v \mu_r + v \mu_r]\,d\leb^{2d}\,dr\\
		&\hspace{2em}\qquad + \gamma\int_0^t\intrr \nabla_v\psi_\delta'(\mu_r/\eta_R)\cdot \sfF(\cdot,\rho_r)\,d\mu_r\,dr \\
		&\hspace{2em}=: \mathrm{(I)} + \mathrm{(II)} + \mathrm{(III)}\,.
	\end{align*}
	Setting $\Psi_\delta(z) := \int_0^s \zeta\psi_\delta''(\zeta)\,d\zeta$, which satisfies $0\le \Psi_\delta(z) \le z$ and $0\le \Psi_\delta'(z)\le 1$ for all $z\ge 0$, we write
	\[
		(\mu_t/\eta_R)\nabla \psi_\delta'(\mu_t/\eta_R)  = \nabla \Psi_\delta(\mu_t/\eta_R)\,.
	\]
	For the first term $\mathrm{(I)}$, we integrate by parts to obtain
	\[
		\int_0^t\intrr v\cdot \nabla_x \psi_\delta'(\mu_r/\eta_R)\,d\mu_r\,dr = \int_0^t\iint_{\{H<R\}} v\cdot (x+v/\gamma)\Psi_\delta(\mu_r/\eta_R)\,d\eta_R\,dr\,,
	\]
	As for $\mathrm{(II)}$ we write
	\begin{align*}
		\mathrm{(II)} &= -\gamma^2\int_0^t\intrr \psi_\delta''(\mu_r/\eta_R)\,\eta_R^{-1}\,(\nabla_v \mu_r -  \mu_r\nabla_v\log\eta_R)\cdot(\nabla_v \mu_r + v \mu_r)\,d\leb^{2d}\,dr \\
		&= -\gamma^2\int_0^t\intrr \psi_\delta''(\mu_r/\eta_R)\,\eta_R^{-1}|\nabla_v \mu_r + v\mu_r|^2\,d\leb^{2d} \\
		&\qquad+ \gamma^2\int_0^t\iint_{\{H\ge R\}} \Psi_\delta'(\mu_r/\eta_R)\,v\cdot(\nabla_v \mu_r + v\mu_r)\,d\leb^{2d}\,dr \\
		&\qquad-\gamma\int_0^t\iint_{\{H<R\}} \Psi_\delta'(\mu_r/\eta_R)\,(x+v/\gamma)\cdot(\nabla_v \mu_r + v\mu_r)\,d\leb^{2d}\,dr\,.
	\end{align*}
	For term $\mathrm{(III)}$, we reformulate to obtain
	\begin{align*}
		\mathrm{(III)} &= \gamma\int_0^t\intrr \psi_\delta''(\mu_r/\eta_R)\,\eta_R^{-1}\,(\nabla_v \mu_r - \nabla_v\log \eta_R)\cdot \sfF(\cdot,\rho_r)\,d\mu_r\,dr \\
		&= \gamma\int_0^t\intrr \psi_\delta''(\mu_r/\eta_R)\,\eta_R^{-1}\,(\nabla_v \mu_r + v\mu_r)\cdot \sfF(\cdot,\rho_r)\,d\mu_r\,dr \\
		&\qquad- \gamma\int_0^t\iint_{\{H\ge R\}} \Psi_\delta'(\mu_r/\eta_R)\,v\cdot \sfF(\cdot,\rho_r)\,d\mu_r\,dr \\
		&\qquad +\int_0^t\iint_{\{H<R\}} \Psi_\delta'(\mu_r/\eta_R)\,(x+v/\gamma)\cdot \sfF(\cdot,\rho_r)\,d\mu_r\,dr\,.
	\end{align*}
	We now 
	collect the terms into
	\begin{align*}
		\Sigma_{R,\delta}(t) &:=\gamma\int_0^t\iint_{\{H<R\}} v\cdot (x+v/\gamma)\Psi_\delta(\mu_r/\eta_R)\,d\eta_R\,dr \\
		&\qquad-\gamma\int_0^t\iint_{\{H<R\}} \Psi_\delta'(\mu_r/\eta_R)\,(x+v/\gamma)\cdot(\nabla_v \mu_r + v\mu_r)\,d\leb^{2d}\,dr \\
		&\qquad+\int_0^t\iint_{\{H<R\}} \Psi_\delta'(\mu_r/\eta_R)\,(x+v/\gamma)\cdot \sfF(\cdot,\rho_r)\,d\mu_r\,dr\,,
	\end{align*}
	and
	\[
		\Lambda_{R}(t) := 
		\gamma\int_0^t\iint_{\{H\ge R\}} |v||\sfF(\cdot,\rho_r)|\,d\mu_r\,dr\,.
	\]
	As for the other terms, we apply the Young's inequality to deduce
	\begin{align*}
		&\intrr \psi_\delta(\mu_t/\eta_R)\,d\eta_R + c_1'\gamma^2\int_0^t\intrr \psi_\delta''(\mu_r/\eta_R)\,\eta_R^{-1}|\nabla_v \mu_r + v\mu_r|^2\,d\leb^{2d}\,dr \\
		&\hspace{8em}\le \intrr \psi_\delta(\mu_0/\eta_R)\,d\eta_R + \Sigma_{R,\delta}(t) + \Lambda_{R}(t)  \\
		&\hspace{8em}\qquad + c_2'\int_0^t\intrr |\sfF(\cdot,\rho_r)|^2\,d \mu_r\,dr 
		+ c_3'\gamma^2\int_0^t\iint_{\{H\ge R\}} |v|^2\,d\mu_r\,dr\,,
	\end{align*}
	for appropriate constants $c_1',c_2',c_3'>0$ that are independent of $\gamma$. The moment estimates given in \eqref{thm_ext_rkin:est} by Theorem~\ref{thm_ext_rkin} allows us to apply the Lebesgue dominated convergence to obtain
	\[
		\lim_{R\to\infty}\lim_{\delta\to 0} \Sigma_{R,\delta} = \Sigma(t)\,, 
	\]
	with
	\[
		\Sigma(t) = \int_0^t\intrr (x+v/\gamma)\cdot \sfF(\cdot,\rho_r)\,d\mu_r\,dr - \gamma\int_0^t\intrr (x+v/\gamma)\cdot\nabla_v \mu_r\,d\leb^{2d}\,dr\,.
	\]
	Notice that the second term in $\Sigma$ may be integrated by parts to obtain
	\begin{align*}
		\intrr (x+v/\gamma)\cdot \nabla_v\mu_r \,d\leb^{2d} \ge - \frac{d}{\gamma} \,.
	\end{align*}
	Indeed, for any $R>0$, we find
	\begin{align*}
		\iint_{\{|x+v/\gamma|<R\}} (x+v/\gamma)\cdot \nabla_v\mu_r \,d\leb^{2d} 
		&\ge -\frac{d}{\gamma}\mu_r(\{|x+v/\gamma|<R\}) \ge - \frac{d}{\gamma}\,,
	\end{align*}
	since $\mu_r(\{|x+v/\gamma|<R\})\le 1$ for every $r\in(0,t)$. Hence, the previous inequality and 
	\begin{align*}
		\intrr |x+v/\gamma|(c_\Phi|x| + c_0)\,d\mu_r &\le \intrr |x+v/\gamma|(c_\Phi|x+v/\gamma| + (c_\Phi/\gamma)|v| + c_0)\,d\mu_r \\
		&\le c_4\intrr H(x,v)\,d\mu_r +c_5\,,
	\end{align*}
	for some constants $c_4,c_5>0$ independent of $\e>0$ and $\gamma\ge 1$, provide the upper bound
	\begin{align*}
		\Sigma(t) &\le c_5'\,t + c_4\int_0^t\intrr H(x,v)\,d\mu_r\,dr\,, 
	\end{align*}
	with $c_5'=d + c_5$. As for the other terms, we apply Fatou's lemma and the Lebesgue dominated convergence to obtain
	\begin{align*}
		\Ent(\mu_t|\Gamma^{-\gamma}_\#\scrN^{2d}) &\le \liminf_{R\to\infty}\intrr \psi(\mu_t/\eta_R)\,d\eta_R\,,\\
		\Ent(\mu_0|\Gamma^{-\gamma}_\#\scrN^{2d}) &= \lim_{R\to\infty}\lim_{\delta\to 0}\intrr \psi_\delta(\mu_0/\eta_R)\,d\eta_R\,, \\
		\int_0^t\mathcal{D}(\mu_r)\,dr &\le \liminf_{\delta\to 0} \int_0^t\intrr \psi_\delta''(\mu_r/\eta_R)\,\eta_R^{-1}|\nabla_v \mu_r + v\mu_r|^2\,d\leb^{2d}\,dr\,.
	\end{align*}
	Putting the inequalities together yields
	\begin{align*}
		&\Ent(\mu_t|\Gamma^{-\gamma}_\#\scrN^{2d}) + c_1'\gamma^2\int_0^t\mathcal{D}(\mu_r)\,dr \\
		&\hspace{6em}\le \Ent(\mu_0|\Gamma^{-\gamma}_\#\scrN^{2d}) + c_5'\,t + c_4\int_0^t\intrr H(x,v)\,d\mu_r\,dr\,. 
	\end{align*}
	Making use of \eqref{eq:aux-rho} to estimate the second term on the right-hand, we obtain
	\[
		\intrr H(x,v)\,d\mu_r \le 4\Bigl(\Ent(\mu_r|\Gamma^{-\gamma}_\#\scrN^{2d}) + d\log 2\Bigr)\,.
	\]
	In particular, we obtain constants $\alpha,\beta,\lambda>0$ independent of $\e>0$ and $\gamma\ge 1$ such that
	\begin{align*}
		&\Ent(\mu_t|\Gamma^{-\gamma}_\#\scrN^{2d}) + \alpha\gamma^2\int_0^t\mathcal{D}(\mu_r)\,dr \\
		&\hspace{6em}\le \Ent(\mu_0|\Gamma^{-\gamma}_\#\scrN^{2d}) + \lambda\int_0^t \Ent(\mu_r|\Gamma^{-\gamma}_\#\scrN^{2d})\,dr + \beta\,t\,.
	\end{align*}
	By Gr\"onwall's lemma, we futher obtain
    \begin{align}\label{eq:g-estimate-diff-inequality}
        \Ent(\mu_t^\gamma|\Gamma^{-\gamma}_\#\scrN^{2d}) 
        \le M(\mu_0^\gamma,t):=\Ent(\mu_0^\gamma|\Gamma^{-\gamma}_\#\scrN^{2d})e^{\lambda t} + (\beta/\lambda)\bigl( e^{\lambda t}-1\bigr).
    \end{align}
    Therefore, integrating \eqref{eq:g-estimate-diff-inequality} and using the previous inequality yield
    \begin{align*}
       \Ent(\mu_t^\gamma|\Gamma^{-\gamma}_\#\scrN^{2d}) + \alpha\gamma^2\int_0^t \mathcal{D}(\mu_r)\,dr
        \le M(\mu_0^\gamma,t)\,,
    \end{align*}
 	thereby concluding the proof.
	\end{proof}

\begin{remark}
	Notice that in the case when $\nabla\Phi$ has linear growth at infinity and $\nabla K$ is bounded and Lipschitz, the assumption on $\sfF^\e$ in Theorem~\ref{thm:entropy-linear} is trivially satisfied. In particular, the inequality \eqref{eq:entropy-linear} holds also for weak solutions $\mu$ to \eqref{app_rkin}.
\end{remark}

\subsubsection{Uniform $W_{x,H}^{k,p}$-estimates}

Due to the regularity supplied by Theorem~\ref{thm_ext_rkin}, the formal computations made in Section~\ref{sec:uniform-gamma} become rigorous. In particular, we obtain for each $\e>0$ and $\gamma>0$ the following restatement of Proposition~\ref{lem:regularity-summary}:

\begin{lemma}\label{thm:regularity-summary}
	Let $p\in[2,\infty)$, $k\in\{0,1\}$. Further, let $(\mu_0^\e)_{\e>0}\subset \mc_c^\infty(\R^d \times \R^d)$ satisfy
	\[
		\sup_{\e>0}\|\mu_0^\e\|_{W_{x,H^\e}^{k,p}} <\infty\,,\qquad H^\e(x,v) = \Phi^\e(x) + |v|^2/2\,,
	\]
	and $(\mu_t^\e)_{t\in[0,T]}\subset\calP(\R^d\times\R^d)$ the corresponding solutions to \eqref{app_rkin} satisfying \eqref{thm_ext_rkin:est}. If
	\[
		\|\nabla K^\e\star \varrho \|_{W^{k,\infty}} \le c_{K,p}\lt(1 + \|\nu\|_{W_{x,H^\e}^{k,p}}\rt),\qquad \varrho = \pi^x_\#\nu,\qquad \text{$\nu\in W_{x,H^\e}^{k,p}(\R^d\times\R^d)\,,$}
	\]
	for some constant $c_{K,p}>0$ independent of $\e$, then there exists some $T_p\in(0,T]$ such that
	\begin{align}\label{eq:regularity-summary}
		\|\mu_t^\e\|_{W_{x,H^\e}^{k,p}} \le M_k\qquad\text{for almost every $t\in[0,T_p]$\,,}
	\end{align}
	where $M_k>0$ is a constant independent of $\e>0$ and $\gamma>0$.
	
	In particular, we have the estimate
	\[
		\sup\nolimits_{t\in[0,T_p]}\|\nabla K^\e\star \rho_t^\e \|_{W^{k,\infty}} \le c_{K,p}(1 + M_k)\,,
	\]
	independently of $\e>0$ and $\gamma>0$.
\end{lemma}

\subsection{Compactness in $\e>0$}

The results obtained in the previous subsection provide compactness in appropriate spaces for the family of solutions $(\mu^\e)_{\e>0}$ to the rescaled equation \eqref{app_rkin}.

\begin{lemma}\label{lem:compact_mu}
	Let $p\in[2,\infty)$ and the assumptions of Theorem~\ref{thm:entropy-linear} and Lemma~\ref{thm:regularity-summary} be satisfied. In particular, the family of solutions
	$(\mu^\e)_{\e>0}$ of \eqref{app_rkin}  satisfies inequalities \eqref{eq:entropy-linear} and \eqref{eq:regularity-summary}.
	
	Then there exists a curve $\mu\in \calC([0,T_p];\calP(\R^d\times\R^d))\cap L^\infty([0,T_p];L^p(\R^d\times\R^d))$ such that
\bq\label{mu_conv}
	\mu_t^\e \rightharpoonup \mu_t\quad\text{weakly in $L^p(\R^d\times\R^d)$ for all $t\in[0,T_p]$}
\eq
for an appropriate subsequence (not relabelled) as $\e\to 0$.

\end{lemma}
\begin{proof}
	Since $H^\e\ge 0$, we also obtain from Lemma~\ref{thm:regularity-summary} the uniform estimate
\[
	\sup_{\e>0}\sup_{t\in[0,T_p]} \|\mu_t^\e\|_{L^p} <\infty\,.
\]
Hence, the sequence $(\mu_t^\e)_{\e>0}\subset L^1(\R^d\times\R^d)$ is equi-integrable for every $t\in[0,T_p]$. In particular, it is tight in $\calP(\R^d\times\R^d)$ for every $t\in[0,T_p]$.

Now let $(\mathsf{v}_t^\e)_{t\in(0,T_p)}$ be the vector field associated to $\mu^\e$, i.e.\
\[
	\partial_t\mu_t^\e + \nabla\cdot(\mu_t^\e \mathsf{v}_t^\e) = 0\qquad\text{in distributions}.
\]
From Theorem~\ref{thm:entropy-linear}, we deduce that $\mathsf{v}_t^\e \in L^2(\mu_t^\e)$ for almost every $t\in(0,T_p)$, and that
\[
	c_{\mathsf{v}}:=\sup_{\e>0}\int_0^{T_P} \|\mathsf{v}_t^\e\|_{L^2(\mu_t^\e)}^2\,dt <\infty\,,
\]
from which we obtain, for all $\varphi\in\calC_c^\infty(\R^d\times\R^d)$,
\begin{align*}
	\langle\varphi,\mu_t^\e\rangle - \langle\varphi,\mu_s^\e\rangle &= \int_s^t\intrr\langle \nabla\varphi,\mathsf{v}_r^\e\rangle\, d\mu_r^\e \,dr \le \|\nabla\varphi\|_{L^\infty}\int_s^t \|\mathsf{v}_r^\e\|_{L^2(\mu_r^\e)}\,dr\,.
\end{align*}	
Consequently, we find by the Monge--Kantorovich duality that
\[
	\d_1(\mu_t^\e,\mu_s^\e) \le c_{\mathsf{v}}\sqrt{|t-s|}\,,\qquad \text{for every $0\le s<t \le T_p$\,.}
\]
i.e.\ $(\mu^\e)_{\e>0}\subset\calC([0,T_p];\calP(\R^d\times\R^d))$ is equi-continuous, which provides the existence of a curve $\mu\in \calC([0,T_p];\calP(\R^d\times\R^d))$ such that \cite{AGS05}
\[
	\mu_t^\e \to \mu_t\quad\text{narrowly in $\calP(\R^d\times\R^d)$ for all $t\in[0,T_p]$}\,.
\]
Due to the lower semicontinuity of $L^p$-norms with respect to narrow convergence, we then deduce that $\mu_t\in L^p(\R^d\times\R^d)$ for every $t\in[0,T_p]$, and consequently $\mu\in L^\infty([0,T_p];L^p(\R^d\times\R^d))$. It is then possible to extend the previous convergence to the convergence \eqref{mu_conv}.
\end{proof}

As for the first marginal $\rho^\e = \pi^x_\#\mu^\e$ of $\mu^\e$, we obtain a stronger convergence that is needed to pass to the limit in the bilinear term $\langle (K\star\rho^\e)\cdot\nabla\varphi, \mu^\e\rangle$. For this, we make use of a classical result that characterizes strong convergence in $L^p$-spaces (Kolmogorov--Riesz--Fr\'echet theorem).

\begin{proposition}[cf.\ {\cite[Section~4.5]{Brezis2011}}]\label{FK_thm} Let $1 \leq p < \infty$. A family $\mathcal{F} \subset L^p(\R^d)$ is relatively compact if and only if 
\begin{itemize}
\item[(i)] $\mathcal{F}$ is uniformly bounded in $L^p(\R^d)$, i.e., there exists $M > 0$ such that 
\[
\|f\|_{L^p} \leq M \quad \mbox{for all f $\in \mathcal{F}$},
\]
\item[(ii)] $\mathcal{F}$ is equi-integrable in $L^p(\R^d)$, i.e., 
\[
\lim_{R \to \infty} \int_{|x| > R}  |f|^p\,dx = 0 \quad \mbox{uniformly on $\mathcal{F}$},
\]
\item[(iii)] $\mathcal{F}$ is equi-continuous in $L^p(\R^d)$, i.e., 
\[
\lim_{a \to 0} \intr |f(x - a) - f(x)|^p\,dx = 0\quad \mbox{uniformly on $\mathcal{F}$}.
\]
\end{itemize}
\end{proposition}

We apply the above theorem with $f = \rho$ and $p=1$ to obtain the following result:

\begin{lemma}\label{lem:rho_conv}
	Let $p\in[2,\infty)$ and the assumptions of Theorem~\ref{thm:entropy-linear} and Lemma~\ref{thm:regularity-summary} be satisfied. In particular, the family of solutions $(\mu^\e)_{\e>0}$ of \eqref{app_rkin} satisfies inequalities \eqref{eq:entropy-linear} and \eqref{eq:regularity-summary}.
	
	For each $t\in[0,T_p]$, the family of marginals $(\rho_t^\e)_{\e>0}$ of $(\mu_t^\e)_{\e>0}$ is relatively compact in $L^1(\R^d)$. Consequently, there exists a limit curve $\rho\in \calC([0,T_p], \calP(\R^d))\cap L^\infty([0,T_p];L^p(\R^d))$ such that
	\begin{align}\label{rho_conv}
		\rho_t^\e\to \rho_t\quad\text{strongly in\; $L^1(\R^d)$}\quad\text{for almost every $t\in(0,T_p)$}
	\end{align}
	for an appropriate subsequence (not relabeled) as $\e\to 0$. 
\end{lemma}
\begin{proof}
	Due to the mass conservation, we readily see Proposition \ref{FK_thm}(i). We also easily find from the moment estimate of $\mu_t^\e$ in $x$ that
\[
\lim_{R \to \infty}\int_{|x| > R} \rho_t^\e\,dx \leq  \lim_{R \to \infty}\frac{1}{R^2}\sup_{\e>0}\intrr |x|^2 \mu_t^\e\,dxdv = 0,
\]
and this shows that Proposition \ref{FK_thm}(ii) holds for $p=1$.

We now check (iii) in Proposition \ref{FK_thm}. 
Let us introduce the Green function $\mathcal{G}^\gamma _t(x,v,\xi,\nu)$ for the linear Vlasov--Fokker--Planck equation  \cite{Bou93}, which is a fundamental solution of
\[
\pa_t \mathcal{G}^\gamma _t + v \cdot \nabla_x \mathcal{G}^\gamma _t = \gamma \nabla_v \cdot (\nabla_v \mathcal{G}^\gamma _t + v \mathcal{G}^\gamma _t),\qquad \mathcal{G}^\gamma _0(x,v,\xi,\nu) = \delta_{x,v}(\xi,\nu)\,.
\]
Then by employing the Green function $\mathcal{G}^\gamma$, we write our solution of \eqref{app_kin} in the integral form:
\[
\mu_t^\e(x,v) = \intrr \mathcal{G}_t^\gamma(x,v,\xi,\nu) \,\mu_0^\e(d\xi d\nu) +\int_0^t \intrr \nabla_\nu \mathcal{G}_s^\gamma(x,v,\xi,\nu) \cdot \sfF^\e(\xi,\rho_{t-s}^\e)\,\mu_{t-s}^\e(d\xi d\nu)\, ds.
\]
We next introduce 
\[
\mathcal{K}_t^\gamma(x,\xi,\nu) := \intr \mathcal{G}_t^\gamma(x,v,\xi,\nu)\,dv\,,
\]
and obtain
\[
\rho_t^\e(x) = \intrr \mathcal{K}_t^\gamma(x,\xi,\nu) \,\mu_0^\e(d\xi d\nu) + \int_0^t \intrr \nabla_\nu \mathcal{K}_s^\gamma(x,\xi,\nu) \cdot \sfF^\e(\xi,\rho_{t-s}^\e)\,\mu_{t-s}^\e(d\xi d\nu)\, ds,
\]
where $\mathcal{K}_t^\gamma$ and $\nabla_\nu\mathcal{K}_t^\gamma$ can be explicitly expressed as
\begin{align*}
\mathcal{K}_t^\gamma(x,\xi,\nu) &= \frac{1}{(2\gamma d(t))^{d/2}} \scrN^d\lt(\frac{x - \xi - \sigma(t) \nu}{\sqrt{2\gamma d(t)}}\rt),\\
(\nabla_\nu \mathcal{K}_t^\gamma)(x,\xi,\nu) &= \frac{-\sigma(t)}{(2\gamma d(t))^{(d+1)/2}} (\nabla \scrN^d)\lt(\frac{x - \xi - \sigma(t) \nu}{\sqrt{2\gamma d(t)}}\rt),
\end{align*}
with 
\[
\sigma(t) = \frac{1 - e^{-\gamma t}}{\gamma}, \qquad d(t) = \int_0^t \sigma(s)^2\,ds\,.
\]
For any $a \in \R$, we use these observations to estimate
\[
\intr\lt| \mathcal{K}_t^\gamma(x-a,\xi,\nu) -\mathcal{K}_t^\gamma(x,\xi,\nu) \rt|dx = \|\scrN^d(\cdot-a) -\scrN^d\|_{L^1}\,,
\]
and
$$\begin{aligned}
\intr\lt| \nabla_\nu\mathcal{K}_t^\gamma(x-a,\xi,\nu) - \nabla_\nu\mathcal{K}_t^\gamma(x,\xi,\nu) \rt|dx = \frac{\sigma(t)}{(2\gamma d(t))^{1/2}}\|\nabla\scrN^d(\cdot-a) - \nabla\scrN^d\|_{L^1}\,.
\end{aligned}$$
Together, this yields
$$\begin{aligned}
\|\rho_t^\e(\cdot-a) - \rho_t^\e\|_{L^1}
&\leq \|\scrN^d(\cdot-a) -\scrN^d\|_{L^1} \cr
&\hspace{-4em}+ \|\nabla\scrN^d(\cdot-a) - \nabla\scrN^d\|_{L^1} \int_0^t \frac{\sigma(s)}{(2\gamma d(s))^{1/2}}  \lt(\intrr |\sfF^\e(\xi,\rho_{t-s}^\e)|\,\mu_{t-s}^\e(d\xi d\nu)\rt)ds\,.
\end{aligned}$$
We then estimate the second term on the right hand side of the above inequality. Note that 
\[
|\sfF^\e(\xi,\rho_t)| 
\leq c_{\Phi} (1 + |\xi|) + c_0,\qquad c_0 = \sup\nolimits_{t\in [0,T_p]}\|\nabla K^\e \star \rho_t^\e\|_{L^\infty}\,,
\]
and the second moment bound estimate of $\mu_t$ in $x$ gives
\bq\label{bdd_f}
\intrr |\sfF^\e(\xi,\rho_{t-s}^\e)|\,\mu_{t-s}^\e(d\xi d\nu) \leq C\lt(1 + \sup_{\e>0}\intrr |x|^2\,\mu_0^\e(dx dv) \rt).
\eq
Furthermore, by Taylor's theorem, we get
\[
\sigma(s) \simeq s \quad \mbox{and} \quad d(s) \simeq s^3\quad\text{implying}\quad \frac{\sigma(s)}{(2\gamma d(s))^{1/2}} \simeq s^{-1/2}\,.
\]
Combining this with \eqref{bdd_f} deduces 
\[
\int_0^t \frac{\sigma(s)}{(2\gamma d(s))^{1/2}}  \intrr |\sfF^\e(\xi,\rho_{t-s}^\e)|\,d \mu_{t-s}^\e(\xi,\nu)\,ds \leq C t^{1/2},
\]
where $C>0$ is independent of $t\in(0,T_p)$, $a\in\R^d$ and $\e>0$. On the other hand, by the Lebesgue dominated convergence theorem, we easily find
\[
\lim_{a \to 0}\| \scrN^d(\cdot-a) -\scrN^d \|_{L^1} =0= \lim_{a \to 0} \| \nabla\scrN^d(\cdot-a) -\nabla\scrN^d\|_{L^1}\,. 
\]
Thus we have
\[
\lim_{a \to 0 }\|\rho_t^\e(\cdot-a) - \rho_t^\e\|_{L^1} =0\,,
\]
and this asserts Proposition \ref{FK_thm}(iii). Consequently, we find a compact set $K\subset L^1(\R^d)$ for which
\[
	\rho_t^\e \in K\qquad\text{for all $t\in[0,T_p]$ and $\e>0$}\,.
\]

To prove the second part of the statement, we notice that $(\rho_t^\e)$ satisfies
\[
	\partial_t\rho_t^\e + \nabla_x\cdot \lt(\intr v\,\mu_t^\e(x,v)\,\leb^{d}(dv)\rt)=0\qquad\text{in distributions\,,}
\]
which then allows us to obtain the uniform estimate
\[
	\d_1(\rho_{t+h}^\e,\rho_t^\e) \le h\lt(\sup_{\e>0}\sup_{r\in[0,T_p]} \sqrt{\intrr |v|^2\,d\mu_t^\e}\rt)\qquad \text{for all $h\ll 1$,\; $\e>0$\,.}
\]
We can then finally conclude that the family $(\rho^\e)_{\e>0}\subset L^1((0,T_p);L^1(\R^d))$ is relatively compact (cf.\ \cite{RossiSavare1003,Simon1986}). In particular, we obtain a curve $\rho\in \calC([0,T_p];\calP(\R^d))\cap L^1((0,T_p);L^1(\R^d))$ such that the required convergence \eqref{rho_conv} holds. Due to Lemma~\ref{lem:lp-rho-u}, we can further deduce, as in Lemma~\ref{lem:compact_mu}, that $\rho\in L^\infty([0,T_p];L^p(\R^d))$ simply from the lower semicontinuity of $L^p$-norms.
\end{proof}

\subsection{Passing to the limit $\e \to 0$} 

We begin by noting that $\Phi^\e$ converges to $\Phi$ locally uniformly as $\e\to 0$. Hence, also $e^{-\beta H^\e}$ converges to $e^{-\beta H}$ locally uniformly for any $\beta\in\R$. Consequently, for any $\varphi\in\calC_c^\infty(\R^d\times\R^d)$, the convergence \eqref{mu_conv} yields (as $\e\to 0$)
\[
	\langle \varphi\, e^{(1-1/p)H^\e}, \mu_t^\e \rangle \longrightarrow \langle \varphi\, e^{(1-1/p)H}, \mu_t \rangle \qquad\text{for every $t\in[0,T_p]$\,.}
\]
In particular, 
\[
	\text{$\mu_t^\e e^{(1-1/p)H^\e}$ converges narrowly to $\mu_t e^{(1-1/p)H}$ for every $t\in[0,T_p]$,}
\]
and by lower-semicontinuity, the inequalities \eqref{eq:entropy-linear} and \eqref{eq:regularity-summary} hold true also for $\mu_t$ instead of $\mu_t^\e$, i.e.\
\[
	\sup\nolimits_{t\in[0,T_p]}\Bigl(\Ent(\mu_t|\Gamma^{-\gamma}_\#\scrN^{2d}) + \|\mu_t\|_{W^{k,p}_{x,H}}\Bigr) <\infty\qquad\text{uniformly in $\gamma\ge 1$\,.}
\]
Following the arguments made in Remark~\ref{rmk:2m}, we conclude that the previous estimate holds true also for $\Ent(\mu_t|\scrN^{2d})$ in place of $\Ent(\mu_t|\Gamma^{-\gamma}_\#\scrN^{2d})$.

To pass to the limit in our equation \eqref{app_rkin}, it suffices to prove that
\begin{align}\label{conv_dist}
	(\nabla K^\e \star \rho^\e_t) \mu_t^\e \to (\nabla K \star \rho_t) \mu_t\qquad\text{in distributions for almost every $t\in(0,T_p)$\,.}
\end{align}
Indeed, the other terms are linear and the convergence $\nabla\Phi^\e \mu_t^\e\to \nabla\Phi \mu_t$ in distribution follows from the local uniform convergence of $\nabla\Phi^\e$ to $\nabla\Phi$ and the narrow convergence of $\mu_t^\e$ to $\mu_t$.

To prove the convergence \eqref{conv_dist}, we write
\begin{align*}
&\langle\varphi,(\nabla K^\e \star \rho_t^\e) \mu_t^\e - (\nabla K \star \rho_t) \mu_t \rangle\cr
&\qquad = \langle \varphi \nabla K \star \rho_t,\mu_t^\e - \mu_t\rangle + \langle \varphi(\nabla K^\e - \nabla K) \star \rho_t^\e ,\mu_t^\e\rangle + \langle \varphi \nabla K \star (\rho_t^\e - \rho_t),\mu_t^\e\rangle \cr
&\qquad =: \mathrm{(I)} + \mathrm{(II)} + \mathrm{(III)}\,.
\end{align*}
Due to the weak convergence provided by \eqref{mu_conv}, we easily obtain $\textrm{(I)}\to 0$ as $\e\to 0$ .

As for $\textrm{(II)}$, we have for almost every $x\in\R^d$ the estimate
\begin{align*}
	((\nabla K^\e - \nabla K)\star\rho_t^\e) (x) 
	&\le \|\nabla K^\e - \nabla K\|_{L^q(B_R)}\|\rho_t^\e\|_{L^{q'}} + \int_{\R^d\setminus B_R(x)} |\nabla K^\e - \nabla K|(x-y)\,\rho_t^\e(dy)\,.
\end{align*}
The first term on the right-hand side converges to zero simply due to the properties of mollification and \eqref{thm:eq:main-app}, while the second terms converges to zero due to the Lebesgue dominated convergence. Together, we obtain
\[
	((\nabla K^\e - \nabla K)\star\rho_t^\e) (x) \longrightarrow 0\qquad\text{as $\e\to 0$ for almost every $x\in\R^d$\,.}
\]
Another application of the Lebesgue dominated convergence then yields $\textrm{(II)}\to 0$ as $\e\to 0$.

%

Now it remains to show the convergence $\mathrm{(III)} \to 0$. For this, we rewrite
\[
	\mathrm{(III)} = \intr \lt( \intr \nabla K(x-y)\rho_t^{\varphi,\e}(x)\,\leb^d(dx)  \rt) (\rho_t^\e - \rho_t)(y)\,\leb^d(dy)\,,
\]
where we set
\[
	\rho_t^{\varphi,\e}(x):= \intr \varphi(x,v) \mu_t^\e(x,v)\,\leb^d(dv)\qquad\text{for almost every $t\in(0,T_p)$}\,.
\]
Due to the strong convergence provided by \eqref{rho_conv}, it suffices to show that
\begin{align}\label{eq:III-estimate}
	\sup_{\e>0}\lt\|\intr \nabla K(x-\cdot)\rho_t^{\varphi,\e}(x)\,\leb^d(dx)\rt\|_{L^\infty} <\infty\,.
\end{align}
However, it turns out that for any $\varphi\in\calC_c^\infty(\R^d\times\R^d)$, 
\[
	\|\rho_t^{\varphi,\e}\|_{L^p} \le \lt(\sup\nolimits_{x\in\R^d}\|\varphi(x,\cdot)\|_{L^{p'}}\rt)\|\mu_t^\e\|_{L^p}\,,
\]
which then gives the required estimate \eqref{eq:III-estimate} due to \eqref{thm:eq:main-app}, and thereby the required convergence \eqref{conv_dist}. All in all, we conclude that our limiting curve $\mu$ is a weak solution of \eqref{app_kin}.


%
%
%
%

\bibliographystyle{myplain}
\bibliography{overdamped-limit}

\end{document}